\documentclass[preprint,14pt]{elsarticle}

\usepackage{caption}
\usepackage{latexsym}
\usepackage{graphicx,subfig}
\usepackage{pgf}
\usepackage{enumitem}
\usepackage{amsmath}
\usepackage{amssymb}
\usepackage{amsthm}
\usepackage{multirow}
\usepackage{color}
\usepackage{url}
\usepackage{longtable}

\def\zhou#1 {\fbox {\footnote {\ }}\ \footnotetext { From Yue: {\color{red}#1}}}

\newtheorem{theorem}{Theorem}[section]

\newtheorem{remark}[theorem]{Remark}
\newtheorem{lemma}[theorem]{Lemma}
\newtheorem{example}{Example}
\newtheorem{corollary}[theorem]{Corollary}
\newtheorem{proposition}[theorem]{Proposition}
\newtheorem{problem}{Problem}
\newtheorem{conjecture}[problem]{Conjecture}

\newcommand{\Gal}{\mathrm{Gal}}
\newcommand{\Z}{\mathbb Z}
\newcommand{\Q}{\mathbb Q}
\newcommand{\F}{\mathbb F}
\newcommand{\K}{\mathbb K}
\newcommand{\Tr}{\mathrm{Tr}}
\newcommand{\RN}[1]{%
	\textup{\uppercase\expandafter{\romannumeral#1}}%
}
\begin{document}
	
\begin{frontmatter}
\title{On the nonexistence of lattice tilings of $\Z^n$ by Lee spheres}
\author[tao]{Tao Zhang}
\ead{taozhang@gzhu.edu.cn}
\author[yue]{Yue Zhou\corref{cor1}}
\ead{yue.zhou.ovgu@gmail.com}
\cortext[cor1]{Corresponding author}

\address[tao]{School of Mathematics and Information Science, Guangzhou University\\ Guangzhou 510006, China.}
\address[yue]{College of Liberal Arts and Sciences, National University of Defense Technology\\ Changsha 410073, China.}

\begin{abstract}
In 1968, Golomb and Welch conjectured that $\mathbb{Z}^{n}$ cannot be tiled by Lee spheres with a fixed radius $r\ge2$ for dimension $n\ge3$. This conjecture is equivalent to saying that there is no perfect Lee codes in $\mathbb{Z}^{n}$ with radius $r\ge2$ and dimension $n\ge3$. Besides its own interest in discrete geometry and coding theory, this conjecture is also strongly related to the degree-diameter problems of abelian Cayley graphs. Although there are many papers on this topic, the Golomb-Welch conjecture is still far from being solved. In this paper, we introduce some new algebraic approaches to investigate the nonexistence of lattice tilings of $\Z^n$ by Lee spheres, which is a special case of the Golomb-Welch conjecture. Using these new methods, we show the nonexistence of lattice tilings of $\Z^n$ by Lee spheres of the same radius $r=2$ or $3$ for infinitely many values of the dimension $n$. In particular, there does not exist lattice tilings of $\Z^n$ by Lee spheres of radius $2$ for all $3\le n\le 100$ except 8 cases.
\end{abstract}

\begin{keyword}
	Perfect Lee code \sep Lattice tiling \sep Golomb-Welch conjecture \sep  Degree-diameter problem
	
	\MSC[2010] 11H31 \sep 11H71 \sep 52C22
\end{keyword}
\end{frontmatter}
\section{Introduction}
50 years ago, Golomb and Welch \cite{GW68} proposed a conjecture on the existence of tilings of $\Z^n$ by Lee spheres of given radius. As summarized in a very recent survey \cite{HK18}, this conjecture is still far from being solved despite of many efforts and papers on it. Besides its own interest in discrete geometry, this conjecture is also strongly related to several different topics in mathematics. For instance, this conjecture is equivalent to the existence problem of perfect codes with respect to the Lee distance which is one of the central research topics in coding theory. It is also strongly related to the degree-diameter problems in graph theory.

First let us introduce several basic concepts and notations. Throughout this paper, let $\mathbb{Z}$ and $\mathbb{Z}_{m}$ be the ring of integers and the ring of integers modulo $m$, respectively. For any two words $x=(x_{1},\dots,x_{n})$ and $y=(y_{1},\dots,y_{n})$ in $\Z^n$ or $\Z_m^n$, their \emph{Lee distance} is defined by
\begin{align*}
&d_{L}(x,y)=\sum_{i=1}^{n}|x_{i}-y_{i}|\text{ for }x,y\in\mathbb{Z}^{n},\text{ and }\\
&d_{L}(x,y)=\sum_{i=1}^{n}\text{min}(|x_{i}-y_{i}|,m-|x_{i}-y_{i}|)\text{ for }x,y\in\mathbb{Z}_{m}^{n}.
\end{align*}
A \emph{Lee code} $C$ is just a subset of $\mathbb{Z}^{n}$ (or $\mathbb{Z}_{m}^{n}$) endowed by the Lee distance. If $C$ has further the structure of an additive group, then $C$ is called \emph{linear Lee code}. Lee codes have many practical applications, for example, constrained and partial-response channels \cite{RS94}, flash memory \cite{S12} and interleaving schemes \cite{BBV98}.

A Lee code $C$ is an \emph{$r$-error-correcting} code if any two distinct elements of $C$ have distance at least $2r+1$. An $r$-error-correcting Lee code is further called \emph{perfect} if for each $x\in\mathbb{Z}^{n}$ ($x\in\mathbb{Z}_{m}^{n}$), there exists a unique element $c\in C$ such that $d_{L}(x,c)\le r$. A perfect $r$-error-correcting Lee code in $\mathbb{Z}^{n}$ and $\mathbb{Z}_{m}^{n}$ will be denoted by $PL(n,r)$-code and $PL(n,r,m)$-code, respectively.

A geometric way of introducing a perfect Lee code is by means of a tiling. Let $V$ be a subset of $\mathbb{Z}^{n}$ (or $\mathbb{Z}_{m}^{n}$). By a copy of $V$ we mean a translation $V+x=\{v+x:\ v\in V\}$ of $V$, where $x\in\mathbb{Z}^{n}$ (or $x\in\mathbb{Z}_{m}^{n}$, respectively). A collection $\mathfrak{T}=\{V+l:\ l\in L\}$, $L\subseteq\mathbb{Z}^{n}$ (or $L\subseteq\mathbb{Z}_{m}^{n}$, respectively), of copies of $V$ constitutes a tiling of $\mathbb{Z}^{n}$ (or $\mathbb{Z}_{m}^{n}$, respectively) by $V$ if $\mathfrak{T}$ forms a partition of $\mathbb{Z}^{n}$ (or $\mathbb{Z}_{m}^{n}$, respectively). If $L$ further forms a lattice, then $\mathfrak{T}$ is called a \emph{lattice tiling} of $\mathbb{Z}^{n}$. Consider the Lee spheres
\begin{align*}
&S(n,r)=\{x\in\mathbb{Z}^{n}:d_{L}(x,0)=|x_{1}|+\dots+|x_{n}|\le r\},\\
&S(n,r,m)=\{x\in\mathbb{Z}_{m}^{n}:d_{L}(x,0)\le r\}.
\end{align*}
Then a code $C$ is a $PL(n,r)$-code (or $PL(n,r,m)$-code) if and only if $\{S(n,r)+c:c\in C\}$ (or $\{S(n,r,m)+c:c\in C\}$, respectively) constitutes a tiling of $\mathbb{Z}^{n}$ by $S(n,r)$ (or $\mathbb{Z}_{m}^{n}$ by $S(n,r,m)$, respectively). Moreover, $C$ is a linear $PL(n,r)$-code if and only if $\{S(n,r)+c:c\in C\}$ forms a lattice tiling of $\mathbb{Z}^{n}$.

For $m\ge2r+1$, as pointed out in \cite[Proposition 1]{HK18}, there exists a natural bijection between $PL(n,r,m)$-codes and $PL(n,r)$-codes that is a union of cosets of $m\mathbb{Z}^{n}\subset\mathbb{Z}^{n}$, given by taking the image or the inverse image with respect to the projection map $\mathbb{Z}^{n}\mapsto\mathbb{Z}_{m}^{n}$. Hence, in this case $PL(n,r)$-codes contain all information about $PL(n,r,m)$-codes.

In the following, we will restrict ourselves to $PL(n,r)$-codes. In \cite{GW68,GW70}, Golomb and Welch constructed $PL(1,r)$-codes, $PL(2,r)$-codes and $PL(n,1)$-codes explicitly. In the same paper, they also proposed the following conjecture.
\begin{conjecture}[Golomb-Welch conjecture]
For $n\ge3$ and $r\ge2$, there does not exist $PL(n,r)$-code.
\end{conjecture}

Although there are many papers related to this topic, this conjecture is far from being solved. We refer to \cite{HK18} for a recent survey on it.

In \cite{GW70}, Golomb and Welch proved the nonexistence of $PL(n,r)$-codes for given $n$ and $r\geq r_{n}$, where $r_{n}$ has not been specified. Later, Post \cite{P75} proved that there is no linear $PL(n,r)$-code for $r\geq\frac{\sqrt{2}}{2}n-\frac{1}{4}(3\sqrt{2}-2)$ and $n\geq6$. In \cite{L81}, Lepist{\"o} showed that a perfect Lee code must satisfy $n\geq(r+2)^{2}/2.1$, where $r\geq285$.

Other researchers have considered the conjecture for small dimensions. In \cite{GMP98}, Gravier et al.\ settled the Golomb-Welch conjecture for 3-dimensional Lee space. Dimension 4 was studied by {\v{S}}pacapan in \cite{S07} with the aid of computer. It was proved in \cite{H09E} that there are no perfect Lee codes for $3\leq n\leq5$ and $r>1$. Horak \cite{H09} showed the nonexistence of perfect Lee codes for $n=6$ and $r=2$.

A special case of the Golomb-Welch conjecture, the nonexistence of linear $PL(n,r)$-codes has also been considered. In \cite{HA12}, the authors established the following connection between lattice tilings of $\mathbb{Z}^{n}$ and finite abelian groups.
\begin{theorem}\cite{HA12}\label{thm1}
Let $S\subseteq\mathbb{Z}^{n}$ such that $|S|=m$. There is a lattice tiling of $\mathbb{Z}^{n}$ by translates of $S$ if and only if there are both an abelian group $G$ of order $m$ and a homomorphism $\phi:\mathbb{Z}^{n}\mapsto G$ such that the restriction of $\phi$ to $S$ is a bijection.
\end{theorem}
If the size of $S$ is a prime, then we have a much stronger result.
\begin{theorem}\cite{S98}\label{thm1_prime}
	Let $S\subseteq\mathbb{Z}^{n}$ such that $|S|=p$ is a prime. There is a tiling of $\mathbb{Z}^{n}$ by translates of $S$ if and only if there exits a homomorphism $\phi:\mathbb{Z}^{n}\mapsto C_p$ that restricts to a bijection from $S$ to the cyclic group $C_p$.
\end{theorem}

By Theorems \ref{thm1} and \ref{thm1_prime}, for $PL(n,r)$-codes, it is not difficult to get the following result.
\begin{corollary}\label{coro1}
There is a linear $PL(n,r)$-code if and only if there are both an abelian group $G$ and a homomorphism $\phi:\mathbb{Z}^{n}\mapsto G$ such that the restriction of $\phi$ to $S(n,r)$ is a bijection. Furthermore, if $|G|$ is a prime, then we do not need the assumption that a $PL(n,r)$-code is linear in the previous statement.
\end{corollary}

\begin{example}
	For $n=2$, Golomb and Welch \cite{GW70} introduced a construction of perfect Lee codes of radius $r$. In particular, when $r=2$, the corresponding abelian group in Corollary \ref{coro1} is $G=C_{13}$ which is the cyclic group of order $13$. Note that each homomorphism $\phi:\mathbb{Z}^{n}\mapsto G$ is determined by the values of $\phi(e_{i})$ for $i=1,\cdots,n$, where $\{e_{i}: i=1,\cdots,n\}$  is the standard basis of $\mathbb{Z}^{n}$. Here we may take $\phi(e_1)=1$ and $\phi(e_2)=5$. If we look at the set $\{\pm x \pm y: x,y=0,1,5 \}$, we see every element of $G$ in it. In Figure \ref{fig:C_13}, we illustrate the perfect Lee code associated with $\{1,5\}\subseteq C_{13}$.
	\begin{figure}[h!]
		\centering
		\includegraphics[width=0.7\textwidth]{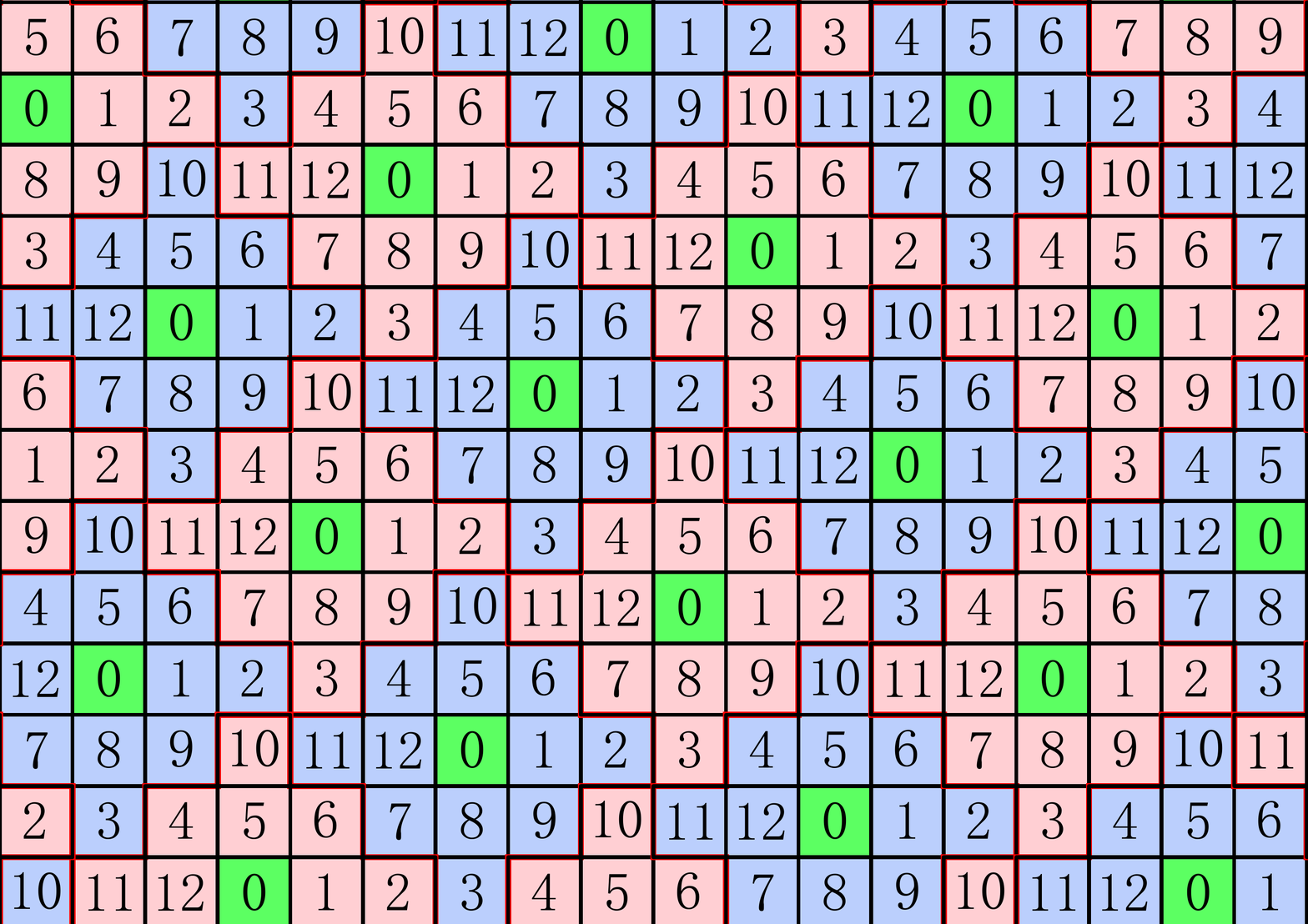}
		\caption{The $PL(2,2)$-code generated by $\{1,5\}\subseteq C_{13}$.}
		\label{fig:C_13}
	\end{figure}
\end{example}

Based on Corollary \ref{coro1}, some nonexistence results for linear $PL(n,r)$-codes have been given. In \cite{HG14}, Horak and Gr\'osek obtained the nonexistence of linear $PL(n,2)$-codes for $7\leq n\leq 12$. In a recent paper by the first author and Ge \cite{ZG17}, the nonexistence of linear perfect Lee codes of radii $3$ and $4$ are proved for several dimensions. By using a polynomial method, Kim \cite{K17} has achieved an important progress showing that there is no $PL(n,2)$-codes for a certain class of $n$ which is expected to be infinite, provided that $2n^2+2n+1$ is a prime satisfying certain conditions. Very recently, Qureshi, Campello and Costa \cite{QCC18} improved this result by considering the projection of the bijection $\phi$ in Corollary \ref{coro1} and some results of friendly primes, which shows that there is no linear $PL(n,2)$-codes for infinitely many $n$.

As pointed out in \cite{HK18}, it seems that the most difficult case of the Golomb-Welch conjecture is that of $r=2$. On one hand, the case $r=2$ is a threshold case as there is a $PL(n,1)$-code for all $n$. On the other hand, the proof of nonexistence of $PL(n,r)$-codes for $3\le n\le5$ and all $r\ge2$ is based on the nonexistence of $PL(n,2)$-codes for the given $n$ \cite{H09E}.

In this paper, our main results are the nonexistence of linear $PL(n,r)$-codes for $r=2,3$ and infinitely many $n$ obtained via the group ring approach. In Section~\ref{pre}, we first give a connection between linear perfect Lee codes and degree-diameter problem, then we provide a group ring representation of linear perfect Lee codes of radii 2 and 3. In Sections~\ref{radius2}, after presenting a generalization of the result obtained recently by Kim \cite{K17}, we turn to the group ring approach to derive several necessary conditions on the existence of linear $PL(n,2)$-codes. By them, we show that linear $PL(n,2)$-codes do not exist for infinitely many $n$. In particular, there is no linear $PL(n,2)$-codes in $\Z^n$ for all $3\le n\le 100$ except 8 cases; see Table \ref{table:100} in Appendix. In Section \ref{radius3}, we prove that linear $PL(n,3)$-codes do not exist for infinitely many values of $n$. Section \ref{sec:conclusion} concludes this paper.

\section{Preliminaries}\label{pre}
\subsection{Linear perfect Lee codes and degree-diameter problem}
 In a graph $\Gamma$, the \emph{distance} $d(u,v)$ from a vertex $u$ to another vertex $v$ is the length of a shortest $u$-$v$ path in $\Gamma$. The largest distance between two vertices in $\Gamma$ is the \emph{diameter} of $\Gamma$. Let $\Gamma=(V,E)$ be a graph of maximum degree $d$ and diameter $k$. According to the famous Moore bound, $\Gamma$ has at most $1+d+d(d-1)+\cdots+d(d-1)^{k-1}$ vertices. When the order of $V$ equals $1+d+d(d-1)+\cdots+d(d-1)^{k-1}$, the graph $\Gamma$ is called a \emph{Moore graph}. Except $k=1$ or $d\le2$, Moore graphs are only possible for $d=3,7,57$ and $k=2$ \cite{D73,HS60}. The graphs corresponding to the first two degrees are the Petersen graph and the Hoffman-Singleton graph. The existence of a Moore graph with degree 57 and diameter 2 is still open. As there are very few Moore graphs, it is interesting to ask the following so-called degree-diameter problem.
\begin{problem}
	Given positive integers $d$ and $k$, find the largest possible number $N(d,k)$ of vertices in a graph with maximum degree $d$ and diameter $k$.
\end{problem}

We refer to \cite{MS13} for a recent survey on the degree-diameter problem. Next, we look at a special type of graphs which is defined by a group.

Let $G$ be a multiplicative group with the identity element $e$ and $S\subseteq G$ such that $S^{-1}=S$ and $e\not\in S$. Here $S^{-1}=\{s^{-1}: s\in S\}$. The Cayley graph $\Gamma(G,S)$ has a vertex set $G$, and two distinct vertices $g,h$ are adjacent if and only if $g^{-1}h\in S$. Here $S$ is called the \emph{generating set}. In particular, when $G$ is abelian, we call $\Gamma(G,S)$ an \emph{abelian Cayley graph}. The following result is not difficult to prove.
\begin{proposition}\label{prop1}
	The diameter of a Cayley graph $\Gamma(G,S)$ is $k$ if and only if $k$ is the smallest integer such that all elements in $G$ appear in $\{\Pi_{i=1}^{k}s_{i}:s_{i}\in S\cup\{e\}\}$.
\end{proposition}

By Corollary~\ref{coro1}, there is a linear $PL(n,r)$-code if and only if there are both an abelian group $G$ and a homomorphism $\phi:\mathbb{Z}^{n}\mapsto G$ such that the restriction of $\phi$ to $S(n,r)$ is a bijection. Note that each homomorphism $\phi:\mathbb{Z}^{n}\mapsto G$ is determined by the values of $\phi(e_{i}),\ i=1,\dots,n$, where $e_{i},\ i=1,\dots,n$, is the standard basis of $\mathbb{Z}^{n}$. Let $D:=\{\pm\phi(e_{i}): i=1,\dots,n\}$, then $\{\Pi_{i=1}^{r}d_{i}:d_{i}\in D\cup\{e\}\}=G$. Hence,
by Proposition~\ref{prop1}, there exists a linear $PL(n,r)$-code if and only if there exists an abelian Cayley graph with degree $2n$, diameter $r$ and $|S(n,r)|$ vertices. Note that $|S(n,r)|=\sum_{i=0}^{\text{min}\{n,r\}}2^{i}\binom {n}{i}\binom {r}{i}$ which is proved in \cite{stanton_note_1970}. This link was also pointed out by Camarero and Mart\'inez \cite{camarero_quasi-perfect_lee_2016}. Let $AC(\Delta,k)$ be the largest order of abelian Cayley graphs of degree $\Delta$ and diameter $k$. Then Golomb-Welch conjecture implies the following conjecture.
\begin{conjecture}
For $k\ge2$ and $d\ge3$, $AC(2d,k)<\sum_{i=0}^{\text{min}\{k,d\}}2^{i}\binom {k}{i}\binom {d}{i}.$
\end{conjecture}

The right-hand-side of the above inequality is also called the \emph{abelian Cayley Moore bound} for abelian groups with $d$-elements generating sets. This upper bound was also pointed by Dougherty and Faber in \cite{dougherty_degree-diameter_2004}, in which they investigated the upper and lower bounds of $AC(2d,k)$ by considering the associated lattice tilings of $\Z^n$ by Lee spheres. Beside giving a better upper bound on $AC(\Delta,k)$, it is also a challenging task to find better/exact lower bounds on $AC(\Delta,k)$ by constructing special Cayley graphs; for recent progress on this topic as well as the same problem for nonabelian Cayley graphs, we refer to \cite{abas_cayley_2016},~\cite{bachraty_asymptotically_2017},~\cite{macbeth_cayley_2012},~\cite{MS13},~\cite{pott_cayley_2017} and \cite{siagiova_approaching_2012}.

\subsection{Linear perfect Lee codes and group ring}
Let $G$ be a finite group (written multiplicatively). The group ring $\mathbb{Z}[G]$ is a free abelian group with a basis $\{g\mid g\in G\}$. For any set $A$ whose elements belong to $G$ ($A$ may be a multiset), we identify $A$ with the group ring element $\sum_{g\in G}a_gg$, where $a_g$ is the multiplicity of $g$ appearing in $A$. Given any $A=\sum_{g\in G} a_gg\in\mathbb{Z}[G]$, we define $A^{(t)}=\sum_{g\in G} a_gg^{t}$. Addition and multiplication in group rings are defined as:
$$\sum_{g\in G}a_{g}g+\sum_{g\in G}b_{g}g=\sum_{g\in G}(a_{g}+b_{g})g,$$
and
$$\sum_{g\in G}a_{g}g\sum_{g\in G}b_{g}g=\sum_{g\in G}(\sum_{h\in G}a_{h}b_{h^{-1}g})g.$$

For a finite abelian group $G$, denote its character group by $\widehat{G}$.
For any $A=\sum_{g\in G} a_g g$ and $\chi \in \widehat{G}$, define $\chi(A)=\sum_{g\in G} a_g \chi(g)$.
The following \emph{inversion formula} shows that $A$ is completely determined
by its character value $\chi(A)$, where $\chi$ ranges over $\widehat{G}$.

\begin{lemma}\label{inversion_formula}
	Let $G$ be an abelian group. If $A=\sum_{g\in G}a_g
	g\in \Z[G]$, then
	\begin{equation}\label{eq:inversion_formula}
		a_{h}=\frac{1}{|G|}\sum_{\chi\in\widehat G}\chi(A) \chi (h^{-1}),	
	\end{equation}
	for all $h\in G$.
\end{lemma}

Group rings and the associated characters are widely used in the research of difference sets and related topics. Most of the important nonexistence results of certain difference sets are obtained by using the group ring language; see \cite{beth_design_1999},~\cite{pott_finite_1995},~\cite{schmidt_cyclotomic_1999} and the references within.

Next, we translate the existence of linear $PL(n,r)$-codes for $r=2$ and $3$ into group ring equations.
\begin{lemma}\label{lemma1}
Let $n\ge2$, then there exists a linear $PL(n,2)$-code if and only if there exist a finite abelian group $G$ of order $2n^{2}+2n+1$ and $T\subseteq G$ viewed as an element in $\mathbb{Z}[G]$ satisfying
\begin{enumerate}
	\item $1\in T$,
	\item $T=T^{(-1)}$,
	\item $T^{2}= 2G-T^{(2)} +2n$.
\end{enumerate}
 \end{lemma}
\begin{proof}
	By Corollary~\ref{coro1}, there exists a linear $PL(n,2)$-code if and only if there are both an abelian group $G$ (written multiplicatively) of order $2n^{2}+2n+1$ and a homomorphism $\phi:\mathbb{Z}^{n}\mapsto G$ such that the restriction of $\phi$ to $S(n,2)$ is a bijection. Note that each homomorphism $\phi:\mathbb{Z}^{n}\mapsto G$ is determined by the values of $\phi(e_{i})$ for $i=1,\cdots,n$, where $\{e_{i}: i=1,\cdots,n\}$ is the standard basis of $\mathbb{Z}^{n}$. Hence there exists a linear $PL(n,2)$-code if and only if there exists an $n$-subset $\{a_{1},a_{2},\dots,a_{n}\}\subseteq G$ such that
	\begin{align*}
		\{1\}\bigcup\{a_{i}^{\pm1},a_{i}^{\pm2}: i=1,\dots,n\}\bigcup\{a_{i}^{\pm1}a_{j}^{\pm1}:1\le i<j\le n\}=G.
	\end{align*}
	
	In the language of group ring, the above equation can be written as
	\begin{align*}
	G=1+\sum_{i=1}^{n}(a_{i}+a_{i}^{-1}+a_{i}^{2}+a_{i}^{-2})+\sum_{1\le i<j\le n}(a_{i}+a_{i}^{-1})(a_{j}+a_{j}^{-1}).
	\end{align*}
	Let $T=1+\sum_{i=1}^{n}(a_{i}+a_{i}^{-1})$. Then we can compute to get that
	\begin{align*}
	T^{2}&=(1+\sum_{i=1}^{n}(a_{i}+a_{i}^{-1}))^{2}\\
	     &=1+2\sum_{i=1}^{n}(a_{i}+a_{i}^{-1})+(\sum_{i=1}^{n}(a_{i}+a_{i}^{-1}))^{2}\\
	     &=1+2\sum_{i=1}^{n}(a_{i}+a_{i}^{-1})+\sum_{i=1}^{n}(a_{i}^{2}+a_{i}^{-2})+2\sum_{1\le i<j\le n}(a_{i}+a_{i}^{-1})(a_{j}+a_{j}^{-1})+2n.
	\end{align*}
	Combining the above equations, we can get our result.
\end{proof}
\begin{lemma}\label{lemma2}
Let $n\ge3$, then there exists a linear $PL(n,3)$-code if and only if there exist a finite abelian group $G$ of order $1+6n^{2}+\frac{4n(n-1)(n-2)}{3}$ and $T\subseteq G$ viewed as an element in $\mathbb{Z}[G]$ satisfying
\begin{enumerate}
	\item $1\in T$,
	\item $T=T^{(-1)}$,
	\item $T^{3}= 6G-3T^{(2)}T-2T^{(3)} +6nT$.
\end{enumerate}
\end{lemma}
\begin{proof}
	By Corollary~\ref{coro1}, there exists a linear $PL(n,3)$-code if and only if there are both an abelian group $G$ (written multiplicatively) of order $1+6n^{2}+\frac{4n(n-1)(n-2)}{3}$ and a homomorphism $\phi:\mathbb{Z}^{n}\mapsto G$ such that the restriction of $\phi$ to $S(n,3)$ is a bijection. Note that each homomorphism $\phi:\mathbb{Z}^{n}\mapsto G$ is determined by the values of $\phi(e_{i})$ for $i=1,\cdots,n$, where $\{e_{i}: i=1,\cdots,n\}$ is the standard basis of $\mathbb{Z}^{n}$. Hence there exists a linear $PL(n,3)$-code if and only if there exists an $n$ tuple $(a_{1},a_{2},\dots,a_{n})$ of elements in $G$ such that
	\begin{align*}
	G=&\{1\}\cup\{a_{i}^{\pm1},a_{i}^{\pm2},a_{i}^{\pm3}: i=1,\dots,n\}\cup\{a_{i}^{\pm1}a_{j}^{\pm1}:1\le i<j\le n\}\\
	&\cup\{a_{i}^{\pm2}a_{j}^{\pm1}:1\le i\ne j\le n\}\cup\{a_{i}^{\pm1}a_{j}^{\pm1}a_{k}^{\pm1}:1\le i<j<k\le n\}.
	\end{align*}
	
	In the language of group ring, the above equation can be written as
	\begin{align*}
	G=&1+\sum_{i=1}^{n}(a_{i}+a_{i}^{-1}+a_{i}^{2}+a_{i}^{-2}+a_{i}^{3}+a_{i}^{-3})+\sum_{1\le i<j\le n}(a_{i}+a_{i}^{-1})(a_{j}+a_{j}^{-1})\\
	&+\sum_{1\le i\ne j\le n}(a_{i}^2+a_{i}^{-2})(a_{j}+a_{j}^{-1})+\sum_{1\le i<j<k\le n}(a_{i}+a_{i}^{-1})(a_{j}+a_{j}^{-1})(a_{k}+a_{k}^{-1}).
	\end{align*}
	Let $T=1+\sum_{i=1}^{n}(a_{i}+a_{i}^{-1})$. Then we can compute to get that
	\begin{align*}
	T^{(2)}T=&(1+\sum_{i=1}^{n}(a_{i}^2+a_{i}^{-2}))(1+\sum_{i=1}^{n}(a_{i}+a_{i}^{-1}))\\
	=&1+\sum_{i=1}^{n}(a_{i}^2+a_{i}^{-2})+\sum_{i=1}^{n}(a_{i}+a_{i}^{-1})+(\sum_{i=1}^{n}(a_{i}^2+a_{i}^{-2}))(\sum_{i=1}^{n}(a_{i}+a_{i}^{-1}))\\
	=&1+\sum_{i=1}^{n}(a_{i}^2+a_{i}^{-2})+2\sum_{i=1}^{n}(a_{i}+a_{i}^{-1})+\sum_{i=1}^{n}(a_{i}^3+a_{i}^{-3})+\sum_{1\le i\ne j\le n}(a_{i}^{2}+a_{i}^{-2})(a_{j}+a_{j}^{-1}),
	\end{align*}
	and
	\begin{align*}
	T^{3}=&(1+\sum_{i=1}^{n}(a_{i}+a_{i}^{-1}))^{3}\\
	=&1+3\sum_{i=1}^{n}(a_{i}+a_{i}^{-1})+3(\sum_{i=1}^{n}(a_{i}+a_{i}^{-1}))^{2}+(\sum_{i=1}^{n}(a_{i}+a_{i}^{-1}))^{3}\\
	=&1+3\sum_{i=1}^{n}(a_{i}+a_{i}^{-1})+3\sum_{i=1}^{n}(a_{i}^{2}+a_{i}^{-2})+6\sum_{1\le i<j\le n}(a_{i}+a_{i}^{-1})(a_{j}+a_{j}^{-1})+6n+\\
	 &\sum_{i=1}^{n}(a_{i}+a_{i}^{-1})^3+3\sum_{1\le i<j\le n}((a_{i}+a_{i}^{-1})^2(a_{j}+a_{j}^{-1})+(a_{i}+a_{i}^{-1})(a_{j}+a_{j}^{-1})^{2})+\\
	&6\sum_{1\le i<j<k\le n}(a_{i}+a_{i}^{-1})(a_{j}+a_{j}^{-1})(a_{k}+a_{k}^{-1})\\
	=&1+6n\sum_{i=1}^{n}(a_{i}+a_{i}^{-1})+3\sum_{i=1}^{n}(a_{i}^{2}+a_{i}^{-2})+\sum_{i=1}^{n}(a_{i}^{3}+a_{i}^{-3})+\\
	&6\sum_{1\le i<j\le n}(a_{i}+a_{i}^{-1})(a_{j}+a_{j}^{-1})+6n+3\sum_{1\le i\ne j\le n}(a_{i}^{2}+a_{i}^{-2})(a_{j}+a_{j}^{-1})+\\
	&6\sum_{1\le i<j<k\le n}(a_{i}+a_{i}^{-1})(a_{j}+a_{j}^{-1})(a_{k}+a_{k}^{-1}).
	\end{align*}
	Combining the above equations, we can get our result.
\end{proof}

\section{Radius equals $2$}\label{radius2}
In this part, we present several nonexistence results concerning linear $PL(n,2)$-codes for infinitely many $n$.

The first one is a slight generalization of the main result obtained by Kim in \cite{K17}, which will be given in Section \ref{sub:approach0}. Then we turn to the group ring language and develop a new approach to prove several new nonexistence results.

Before getting into the details of the proof in Sections \ref{sub:approach1} and \ref{sub:approach2}, we would like to provide a sketch of our approach here. To show the nonexistence of linear $PL(n,2)$-codes, we only have to prove there is no $T\subseteq G$ viewed as an element in $\Z[G]$ satisfying the three conditions in Lemma \ref{lemma1}. For each nontrivial group character $\chi\in\widehat{G}$, by applying it on the third condition in Lemma \ref{lemma1}, we obtain
\begin{equation}\label{eq:sketch}
	\chi(T)^2 = -\chi(T^{(2)}) +2n.
\end{equation}
This is actually an equation with two unknown value $\chi(T)$ and $\chi(T^{(2)})$ over $\Z[\zeta_{\exp(G)}]$, where $\exp(G)$ denotes the exponent of $G$ and $\zeta_{\exp(G)}$ is a primitive $\exp(G)$-th root of unity.

At first glance, we do not have any obvious relation between $\chi(T)$ and $\chi(T^{(2)})$. However, by applying $(\cdot)^{(2)}$ on $T^{2}= 2G-T^{(2)} +2n$ recursively,  we get
\begin{align*}
	\chi(T^{(2)})^2 &= -\chi(T^{(2^2)}) +2n,\\
	\chi(T^{(2^2)})^2 &= -\chi(T^{(2^3)}) +2n,\\
	&~~\vdots\\
	\chi(T^{(2^k)})^2 &= -\chi(T) +2n,
\end{align*}
where $k$ is some positive integer. By those equations above, we can eliminate all $\chi(T^{(2^i)})$ for $i=1,2,\cdots, k$ and get an equation only involving $\chi(T)$ as unknown. However, in general we do not know the value of $k$ and the degree of this equation in $\chi(T)$ is very large. Hence it seems elusive to solve this equation directly.

Our first trick is to project the group ring equation $T^{2}= 2G-T^{(2)} +2n$ onto a small quotient group of $G$. For instance, in Theorem \ref{th:r=2_v=5}, we assume that $5$ divides $|G|$ and we consider the image $\overline{T}$ of $T$ in $G/H\cong C_5$ which is the cyclic group of order $5$. In such a small group, we can easily handle $\overline{T}$ and $\overline{T}^{(2)}$, because $\overline{T}^{(4)} = \overline{T}$. More details will be presented in Section \ref{sub:approach1}.

When the smallest prime divisor of $|G|$ is getting larger, even by computer program, it seems impossible to find the univariate polynomial in $\chi(T)$. Hence we can only handle the case in which $5, 13$ or $17 $ divides $|G|$ in Section \ref{sub:approach1}.

Another possible way is to consider our group ring equations modulo $p$ where $p$ is a prime divisor of $n$. Suppose that $-\alpha = \chi(\overline{T})$. Then
\begin{equation*}
\chi(\overline{T}^{(2^i)})\equiv -\alpha^{2^i} \pmod{p}.
\end{equation*}
Consequently we have $\chi(\overline{T}^{(2^ip^j)})\equiv -\alpha^{2^ip^j} \pmod{p}$. For a given $\chi$, if all $\chi((\cdot)^{(2^ip^j)})$ together are exactly the nontrivial characters of $G$, then we may use the inversion formula in Lemma \ref{inversion_formula} to derive the coefficient $a_g$ of each $g$ in $T$. As $a_g\in \{0,1\}$, we get strong restrictions on the value of $\alpha$. In Theorem \ref{th:union_criteria}, we will apply some known results from algebraic number theory and finite fields on them to derive necessary conditions.

\subsection{A generalization of Kim's approach}\label{sub:approach0}
First we generalize the main theorem in \cite{K17}, which is about the nonexistence of perfect $2$-error-correcting Lee codes when $|G|=2n^2+2n+1$ is a prime. Compared with the original one in \cite{K17}, our proof is more or less the same, however we do not need the assumption that $|G|$ is a prime.

\begin{theorem}\label{th:kim_generalization}
	Suppose that $2n^2+2n+1=mv$ where $v$ is a prime and $v>2n+1$. Define
	$a = \min\{ a\in \Z^+: v\mid 4^a+4n+2 \}$ and $b$ is the order of $4$ modulo $v$. (If there is no $a$ with $v\mid 4^a+4n+2$, then we let $a=\infty$.) Assume that there is a linear $PL(n,2)$-code. Then there exists at least one $\ell \in \{0,1,\dots, \lfloor \frac{m}{4}\rfloor \}$ such that the equation
	\[ a(x+1) +by=n-\ell \]
	has nonnegative integer solutions.
\end{theorem}
\begin{proof}
	Within this proof, we let the abelian group $G$ be additive and let $0$ be its identity element.
	
	By Corollary \ref{coro1}, there exists $S=\{s_i: i=1,\dots, n\}\subseteq G$ such that
	\[ \{0\}, \{\pm s_i:  i=1,\dots, n\}, \{\pm2s_i: i=1,\dots, n \}, \{\pm s_i\pm s_j: 1\leq i<j\leq n\}  \]
	form a partition of $G$.
	
	Let $H$ be a subgroup of $G$ of index $v$. Let $\rho:G \rightarrow G/H$ be the canonical homomorphism and $x_i=\rho(s_i)$. Then the multisets
	\[\{0\}, \{*~\pm x_i:  i=1,\dots, n~*\}, \{*~\pm2x_i: i=1,\dots, n ~*\}, \{*~\pm x_i\pm x_j: 1\leq i<j\leq n~*\}  \]
	form a partition of $mG/H$.
	
	As most of the rest part is basically the same as the proof in \cite{K17}, we will omit some details of the computation. Let $k$ be an integer. By calculation,
	\begin{align*}
	&\sum_{i=1}^{n}\left((x_i^{2k} + (-x_i)^{2k} +(2x_i)^{2k} +(-2x_i)^{2k}\right)\\
	&+\sum_{1\leq i < j\leq n}\left( (x_i+x_j)^{2k} + (x_i-x_j)^{2k} + (-x_i+x_j)^{2k} +(-x_i-x_j)^{2k} \right)\\
	=& (4^{k} + 4n + 2)\overline{S}_{2k} + 2\sum_{t=1}^{k-1}\binom{2k}{2t}\overline{S}_{2t}\overline{S}_{2(k-t)}
	\end{align*}
	where $\overline{S}_t := \sum_{i=1}^{n}x_i^t$. Since this is also the sum of the $2k$-th powers of every element in $mG/H$,
	\begin{equation}\label{eq:kim_main}
	(4^{k} + 4n + 2)\overline{S}_{2k} + 2\sum_{t=1}^{k-1}\binom{2k}{2t}\overline{S}_{2t}\overline{S}_{2(k-t)}=
	\begin{cases}
	0,  & v-1 \nmid 2k,\\
	-m, & v-1 \mid 2k.
	\end{cases}
	\end{equation}
	
	Let $a$ and $b$ be the least positive integers satisfying $v\mid 4^a+4n+2$ and $v\mid 4^b-1$. Define
	\[ X = \{ax+by : x\geq 1, y\geq 0\}. \]
	We prove the following two claims by induction on $k$.
	
	\textbf{Claim 1:} If $1\le k < \frac{v-1}{2}$ is not in $X$, then $\overline{S}_{2k}=0$.
	
	Suppose that $\overline{S}_{2k}=0$ for each $k\le k_0-1$ that is not in $X$. Assume that $k_0\notin X$. As $X$ is closed under addition, for any $t$, at least one of $t$ and $k_0-t$ is not in $X$.
	
	For any integer $k$, if $v\mid 4^k+4n+2$, then $k$ must be of the form $a+by$ whence $k\in X$. This implies that $v\nmid 4^{k_0}+4n+2$. By \eqref{eq:kim_main} and the induction hypothesis,
	\[  0=(4^{k_0} + 4n + 2)\overline{S}_{2k_0} + 2\sum_{t=1}^{k_0-1}\binom{2k_{0}}{2t}\overline{S}_{2t}\overline{S}_{2(k_0-t)}=(4^{k_0} + 4n + 2)\overline{S}_{2k_0}.\]
	Thus $\overline{S}_{2k_0}= 0$.
	
	Let $e_k$ be the elementary symmetric polynomials with respect to $x_1^2$, $x_2^2$, $\cdots$, $x_n^2$.
	
	\textbf{Claim 2}: If $1\le k\le n< \frac{v-1}{2}$ is not in $X$, then $e_k=0$.
	
	Suppose that $e_k=0$ for each $k\leq k_0-1$ not in $X$ and $k_0\notin X$. As $X$ is closed under addition, for each $0<t<k_0$, at least one of $t$ and $k_0-t$ is not in $X$. By Claim 1 and the inductive hypothesis, $e_t=0$ or $S_{2(k_0-t)}=0$.
	Together with Newton identities on $x_1^2,\dots, x_n^2$, we have
	\[ k_0e_{k_0} = e_{k_0-1}S_2 + \dots + (-1)^{i+1} e_{k_0-i}S_{2i}+\dots +(-1)^{k_0-1}S_{2k_0}=(-1)^{k_0-1}S_{2k_0}=0.\]
	Therefore $e_{k_0}=0$.
	
 Note that if $x_{i}=0$, then $-x_{i}=2x_{i}=-2x_{i}=0$. As $0$ appears exactly $m$ times in $mG/H$, $0$ appears at most $\lfloor \frac{m}{4} \rfloor$ times in $x_i$'s. Suppose that $0$ appears $\ell$ times in $\overline{S}$. Then $e_{n-l}$ is the product of those nonzero $x_i^2$'s, whence $e_{n-l}\neq 0$. By Claim 2, $n-l$ is in $X$. Therefore, we have finished the proof.
\end{proof}

In general, Theorem \ref{th:kim_generalization} is quite strong. In particular, when $2n^2+2n+1$ is a prime, by Corollary \ref{coro1}, the assumption that $PL(n,2)$-codes are linear in the statement of Theorem \ref{th:kim_generalization} can be removed. This result is exactly what Kim has proved in \cite{K17}. It is not difficult to verify that the nonexistence results in \cite{QCC18} are also covered by Theorem \ref{th:kim_generalization}.

Theorem \ref{th:kim_generalization} provides us new nonexistence results for linear $PL(n,2)$-codes when $|G|$ is not a prime. For instance, when $n=6$, $|G|=85=5\cdot 17$, by computer program we can verify that the necessary condition in Theorem \ref{th:kim_generalization} is not satisfied.

In Table \ref{table:r=2_kim}, we list the cardinality of $n$ for which the existence of linear $PL(n,2)$-codes can be excluded by Theorem \ref{th:kim_generalization}.
\begin{table}[h!]
	\centering
	\begin{tabular}{|c|c|c|c|c|c|}
		\hline
		$N$ 		& $10$ & $10^2$ & $10^3$ & $10^4$ & $10^5$ \\
		\hline
		$\# \{n\leq N: \text{Corollary \ref{coro:r=3_v=7} can be applied}\}$  & $5$ & $68$ & $713$ & $7147$ & $71254$ \\
		\hline
	\end{tabular}
	\caption{The number of $n$ to which Theorem \ref{th:kim_generalization} can be applied.}\label{table:r=2_kim}
\end{table}

\subsection{Nonexistence for $|G|$ with a small prime divisor}\label{sub:approach1}
If we want to use Theorem \ref{th:kim_generalization} to prove the nonexistence of perfect Lee codes, we always require at least one fairly large prime divisor $v$ of $|G|$ to guarantee that $v>2n+1$ and $a(x+1) +by=n-\ell$ has no solutions for each $\ell \in \{0,1,\dots, \lfloor \frac{|G|}{4v}\rfloor \}$. Next we investigate the opposite case in which $|G|$ has fairly small divisors. It is not difficult to see that $5$, $13$ and $17$ are the first $3$ possible prime divisors of $2n^2+2n+1$.

In the rest part of this section, for a positive integer $n$, we always let $G$ be a multiplicative group of order $2n^2+2n+1$ with identity element denoted by $1$. By Lemma~\ref{lemma1}, we consider the existence of $T\subseteq G$ viewed as an element in $\Z[G]$ satisfying
\begin{enumerate}[label=(\alph*)]
	\item \label{condition:main_n=2_0} $1\in T$,
	\item \label{condition:main_n=2_1} $T=T^{(-1)}$,
	\item \label{condition:main_n=2_2} $T^{2}= 2G-T^{(2)} +2n$.
\end{enumerate}
Here $T^{(j)}:= \sum a_g g^j$ for $T=\sum a_g g$. Clearly $T$ contains $2n+1$ elements.

For $\chi\in \widehat{G}$,
\begin{equation}\label{eq:main_n=2_character}
	\chi(T)^2 = \left\{
	\begin{array}{ll}
	-\chi(T^{(2)})+2n, & \chi\neq \chi_0, \\
	(2n+1)^2,& \chi= \chi_0.
	\end{array}
	\right.
\end{equation}

	Let $H$ be a subgroup of $G$ with order $m$ and $\rho: G\rightarrow G/H$ be the canonical homomorphism. For $S=\sum_{g\in G}s_g g$, we define $\overline{S}=\rho(S)= \sum_{g\in G} s_g \rho(g)$. Thus
\[\overline{T}=\sum_{\bar{g}\in G/H}a_{\bar{g}} \bar{g}\in \Z[G/H],\]
where $a_{\bar{g}}=\sum_{\{g: \rho(g)= \bar{g}\}} a_g$. Hence $a_{\bar{g}}\in \{0,\dots,m\}$.
By computation,
\[\overline{S^{(j)}} = \sum_{g\in G} s_g \rho(g^{(j)})= \sum_{g\in G} s_g \rho(g)^{(j)}=\overline{S}^{(j)}.\]
If $T$ satisfies Conditions \ref{condition:main_n=2_1} and \ref{condition:main_n=2_2}, then the following two equations must hold.
\begin{enumerate}[label=(\alph*')]\setcounter{enumi}{1}
	\item \label{condition:main_n=2_1'} $\overline{T}=\overline{T}^{(-1)}$,
	\item \label{condition:main_n=2_2'} $\overline{T}^2= 2mG/H-\overline{T}^{(2)} +2n$.
\end{enumerate}

First, we investigate a very special case for which Conditions \ref{condition:main_n=2_1'} and \ref{condition:main_n=2_2'} hold.
\begin{lemma}\label{lm:SS_r=2}
	Let $S=a+bG\in \Z[G]$ with $|G|=v$. Assume that positive integers $v$ and $m$ satisfy $a+vb=2n+1$ and $mv=2n^2+2n+1$, and $S$ satisfies
	\begin{equation}\label{eq:S^2}
	S^2 = 2mG - S+2n.	
	\end{equation}
 	Then $8n+1$ is a square in $\Z$.
\end{lemma}
\begin{proof}
	By calculation,
	\[ S^2 = ((2n+1)b+ab)G+a^2. \]
	By comparing the coefficient of $1$ in \eqref{eq:S^2}, we get
	\[ a^2+a-2n=0, \]
	which means that $8n+1$ must be a square.
\end{proof}
\begin{remark}
	We can also further compare the coefficient of $G$ in \eqref{eq:S^2}. However, it will not offer us any extra restriction. Moreover, when $c=\sqrt{8n+1}\in \Z$, we can also verify that the condition $v\mid bv = 2n+1-a=\frac{c^2-1}{4}+1-\frac{-1\pm c}{2}=1+\frac{(c\mp 1)^2}{4}$ holds for every divisor $v$ of $2n^2+2n+1$.
\end{remark}

It is straightforward to verify that the smallest possible value of $v$ dividing $2n^2+2n+1$ is $5$. First, let us look at the existence of $\overline{T}$ in $G/H$ which is isomorphic to the cyclic group $C_5$ of order $5$.
\begin{theorem}\label{th:r=2_v=5}
	Assume that $5\mid 2n^2+2n+1$, $8n+1$ is a non-square and $8n-3\neq 5k^2$ for all $k\in \Z$. Then there is no subset $\overline{T} \subseteq C_5$ of size $2n+1$ satisfying Conditions \ref{condition:main_n=2_1'} and \ref{condition:main_n=2_2'}.
\end{theorem}
\begin{proof}
	Assume, by way of contradiction, that $\overline{T}\subseteq C_5$ satisfies Conditions \ref{condition:main_n=2_1'} and \ref{condition:main_n=2_2'}.
	From Condition \ref{condition:main_n=2_2'}, we have
	\[\overline{T}^2 \equiv -\overline{T}^{(2)} + 2n \pmod{G/H}.\]
	Changing $\overline{T}$ to $\overline{T}^{(2)}$, note that $\overline{T}^{(4)}=\overline{T}$, we get that
	\[ (\overline{T}^{(2)})^2 \equiv -\overline{T}^{(4)} + 2n =-\overline{T} + 2n\pmod{G/H}.\]
	Combining above two equations, we have
	\[ \overline{T}^4-4n\overline{T}^2+\overline{T}+4n^2-2n\equiv 0\pmod{G/H}. \]
	By calculation,
	\[\overline{T}^4-4n\overline{T}^2+\overline{T}+4n^2-2n = (\overline{T}^2-\overline{T}-2n+1)(\overline{T}^2+\overline{T}-2n).\]
	As there is no zero divisors in $\Z[G/H]$ and $G/H$ is irreducible in it, one of the two factors must be congruent to $0$ modulo $G/H$.
	
	If the second factor is congruent to $0$ modulo $G/H$, then by a simple counting argument we can show that
	\[\overline{T}^2=2mG/H-\overline{T}+2n.\]
	Together with Condition \ref{condition:main_n=2_2'}, it follows that $\overline{T} = \overline{T}^{(2)}$. As $2$ is primitive modulo $5$, $\overline{T}$ must equal to $a+bG$ for some $a,b\in \Z$. However, by Lemma \ref{lm:SS_r=2}, it contradicts our assumption that $8n+1$ is not a square.
	
	Suppose that $\overline{T}^2-\overline{T}-2n+1\equiv 0 \pmod{G/H}$. Take a non-principle character $\chi\in \widehat{G/H}$, then $\chi(\overline{T})\in \Z[\zeta_5]$ is such that
	\begin{equation}\label{eq:v=5_chi_1}
		 \chi(\overline{T})^2 - \chi(\overline{T}) -2n+1=0,
	\end{equation}
	which means $8n-3$ is a square in $\Z[\zeta_5]$. By checking $8n-3 \pmod{8}$, it is easy to see that $8n-3$ has a square divisor if and only if $8n-3=tk^2$, $t\equiv5\pmod{8}$. Under the condition that $8n-3\neq 5k^2$, by \cite[page 263]{weiss_algebraic_1998}, the smallest cyclotomic field containing $\K=\Q(\sqrt{8n-3})$ is $\Q(\zeta_{d(\K)})$ where $d(\K)=t$ is the discriminant of $\K$. Since $d(\K)=t>5$, there is no $\chi(\overline{T})$ such that \eqref{eq:v=5_chi_1} holds.
	
	Therefore, there is no $\overline{T} \subseteq C_5$ satisfying Conditions \ref{condition:main_n=2_1'} and \ref{condition:main_n=2_2'}.
\end{proof}
\begin{remark}\label{remark:n=5_infinite}
	Theorem \ref{th:r=2_v=5} proves the nonexistence of $\overline{T} \subseteq C_5$ satisfying Conditions \ref{condition:main_n=2_1'} and \ref{condition:main_n=2_2'} for infinitely many $n$. There are many different ways to show it. For instance, it can be readily verified that if $n\equiv 4,5 \pmod{7}$ and $n\equiv 1 \pmod 5$, then all the assumptions in Theorem \ref{th:r=2_v=5} hold.
\end{remark}

To prove further results, we need the following results about the decomposition of a prime $p$ into prime ideals in $\Z[\zeta_w]$ which can be found in \cite{N99}.
\begin{lemma}\label{lm:ideal_decomp}
	Let $p$ be a prime and $\zeta_w$ be a primitive $w$-th root of unity in $\mathbb{C}$. If $w=p^rw'$ with $\gcd(w',p)=1$, then the prime ideal decomposition of $(p)$ in $\Z[\zeta_w]$ is
	$$(p) = (P_1P_2\cdots P_d)^e,$$
	where $P_i$'s are distinct prime ideals, $e=\varphi(p^r)$, $d=\varphi(w')/f$ and $f$ is the order of $p$ modulo $w'$. If $t$ is an integer not divisible by $p$ and $t\equiv p^s \pmod{w'}$ for a suitable integer $s$, then the field automorphism $\sigma_t :\zeta_{w'} \mapsto \zeta_{w'}^t=\zeta_{w'}^{p^s}$  fixes the ideals $P_i$.
\end{lemma}

As $|G|=2n^2+2n+1$, for every prime divisor $p$ of $2n$, $\gcd(p, |G|)=1$ whence the ideal $(p)$ is unramified over $\Z[\zeta_v]$ for any $v\mid 2n^2+2n+1$.

Next, we look at $v=13$ which is the second smallest possible integer dividing $2n^2+2n+1$. Compared with Theorem \ref{th:r=2_v=5}, it is much more complicated.
\begin{theorem}\label{th:r=2_v=13}
	Assume that $13\mid 2n^2+2n+1$, $8n+1$ is not a square and $8n-3\neq 13k^2$ for all $k\in \Z$. Then there is no subset $\overline{T} \subseteq C_{13}$ of size $2n+1$ satisfying Conditions \ref{condition:main_n=2_1'} and \ref{condition:main_n=2_2'}.
\end{theorem}
\begin{proof}
	Assume, by way of contradiction, that there exists $\overline{T} \subseteq C_{13}$ satisfying Conditions \ref{condition:main_n=2_1'} and \ref{condition:main_n=2_2'}, which means that
	\begin{equation}\label{eq:v=13_6equations}
	f_i=\overline{T}^{(2^i)}\overline{T}^{(2^i)} +\overline{T}^{(2^{i+1})} - 2n\equiv 0 \pmod{G/H},
	\end{equation}
	for $i=0,1,\dots, 5$. Regarding them as polynomials with variables $\overline{T}^{(2^i)}$, we compute the resultant of $f_0$ and $f_1$ to obtain a polynomial $h_1$ without $\overline{T}^{(2^1)}$. Then we calculate the resultant $h_2$ of $h_1$ and $f_2$ to eliminate $\overline{T}^{(2^2)}$ and the resultant $h_3$ of $h_2$ and $f_3$ to eliminate $\overline{T}^{(2^3)}$... In the end, we obtain $h_5$ which is with only one variable $\overline{T}$. All these calculations can be done using MAGMA \cite{Magma}. Furthermore, $h_5 \pmod{G/H}$ can be factorized into $4$ irreducible factors
	\[	h_5 \equiv (\overline{T}^2-\overline{T}-2n+1)(\overline{T}^2+\overline{T}-2n) \ell_3 \ell_4 \pmod{G/H},\]
	where
	\begin{align*}
	\ell_3=&\overline{T}^6 - \overline{T}^5 - 6\overline{T}^4n + \overline{T}^4 + 4\overline{T}^3n - \overline{T}^3 + 12\overline{T}^2n^2- 6\overline{T}^2n\\ & + \overline{T}^2 - 4\overline{T}n^2 +	4\overline{T}n - \overline{T} - 8n^3 + 8n^2 - 2n + 1,
	\end{align*}
	and $\ell_4$ is of degree $54$ and has much more terms than $\ell_3$. As $h_5$ is obtained from $f_i$'s, $h_5$ must be congruent to $0$ modulo $G/H$. Moreover, by the same argument used in the proof of Theorem \ref{th:r=2_v=5}, the first two irreducible factors of $h_5$ cannot be congruent to $0$ modulo $G/H$. It must be pointed out that the assumption $8n-3\neq 13k^2$ is necessary here, because we need the condition that the discriminant of $\Q(\sqrt{8n-3})$ is different from $13$. Hence, $\ell_3\ell_4 \equiv 0\pmod{G/H}$.

	Let $\chi\in \widehat{G/H}$ be a non-principal character. It follows that
	\[ \chi(\ell_3 \ell_4) =0. \]
	As $\ell=\ell_3\ell_4$ is too complicated, we do not handle them directly. Instead, we take a prime number $p$ and consider $\chi(\ell_3\ell_4)$ modulo $p$. Let $p$ be primitive modulo $v=13$ which means $p\equiv 2,6,7,11 \pmod{v}$. By Lemma \ref{lm:ideal_decomp}, $(p)$ is a prime ideal in $\Z[\zeta_v]$.
	
	Let us replace $\chi(\overline{T}) \pmod{p}$ by $X$ in $\chi(\ell)\equiv 0\pmod{p}$ and let its coefficients be calculated modulo $p$. Now we obtain a polynomial $\bar{\ell}(X)$ in $\F_{p}[X]$. Then we get that $\bar{\ell}(\tau_1)=0$ for $\tau_1 \equiv \chi(\overline{T})\pmod{p}$.
	
	Since $\overline{T}^{(-1)}=\overline{T}$,  $\chi(\overline{T}^{(-1)})\equiv \chi(\overline{T}) \pmod{p}$ whence we only have to consider the roots of $\bar{\ell}$ in $\F_{p^{(v-1)/2}} = \F_{p^6}$. Note that $\tau_1$ is a root of $\bar{\ell}$. Recall that $\chi(\overline{T}^{(p)})\equiv \chi(\overline{T})^p \pmod{p}$. Hence
	\begin{equation}\label{eq:tau_2i_v=13}
	\chi(\overline{T}^{(2^i)})
	\equiv \tau_i= \left\{
	\begin{array}{ll}
	\tau_1^{p^i} \pmod{p},     & p\equiv 2,11 \pmod{13}, \\
	\tau_1^{p^{6-i}} \pmod{p}, & p\equiv 6,7 \pmod{13}.
	\end{array}
	\right.
	\end{equation}
	
	Let us consider the necessary conditions that $\tau_1$ must satisfy. First, by \eqref{eq:v=13_6equations},
	\begin{equation}\label{eq:v=13_6equations_tau}
		\tau_{i}^2 +\tau_{i+1} - 2n\equiv 0,
	\end{equation}
	for $i=0,1,\cdots, 5$.
	
	Second, we look at the coefficients of $a_{\bar{g}}$ by using the inversion formula \eqref{eq:inversion_formula}. Let $\beta$ be an element of order $v=13$ in $\F_{p^{v-1}}$. For $\bar{g} \in G/H$ with $\chi(\bar{g}) \equiv \beta \pmod{p}$,
	\begin{align}
	\nonumber	a_{\bar{g}} =& \frac{1}{13} \left((2n+1)  + \sum_{i=1}^{12}\chi(\overline{T}^{(i)})\chi(\bar{g}^{-i}) \right)\\
	\nonumber =&\frac{1}{13} \left((2n+1)  + \sum_{i=1}^{6} \chi(\overline{T}^{(i)})(\chi(\bar{g}^{i}) + \chi(\bar{g}^{-i})) \right)\\
	\label{eq:v=13_ag=...}\equiv &
	\dfrac{1}{13} ( 2n+1 + \tau_0(\beta+\beta^{-1}) + \tau_1(\beta^2+\beta^{-2}) + \tau_4(\beta^3+\beta^{-3})\\
	\nonumber &+ \tau_2(\beta^4+\beta^{-4}) + \tau_3(\beta^5+\beta^{-5}) +\tau_5(\beta^6+\beta^{-6})) \pmod{p}.
	\end{align}
	It is clear that $a_{\bar{g}} \pmod{p}$ corresponds to an element in $\F_p$. Moreover, since $|\overline{T}| = |T|=2n+1$, all $a_{\bar{g}}$'s also satisfy that
	\begin{equation}\label{eq:v=13_ag_sum}
	\sum_{\bar{g}\in G/H} a_{\bar{g}} \equiv 2n+1\pmod{p}.
	\end{equation}
	
	For each $p$, the parameter $n$ is not constant modulo $p$.  Depending on the value of $n$ modulo $p$, we have to divide our calculations into the $p$ cases. In each case, $\bar{l}$ is a concrete polynomial. First, we calculate all the roots of $\bar{l}$ in $\F_{p^{v-1}}$. Then for each root $\tau_1$, plug it into \eqref{eq:tau_2i_v=13} to get $\tau_i$ and  check whether \eqref{eq:v=13_6equations} holds for each $i$ and whether $a_{\bar{g}}$ derived from \eqref{eq:v=13_ag=...} satisfying $a_{\bar{g}}\pmod{p} \in \F_p$ and \eqref{eq:v=13_ag_sum}. By our MAGMA program, taking $p=11$, we can show that at least one of the necessary conditions is not satisfied, which means there is no $\overline{T}$ such that Conditions \ref{condition:main_n=2_1'} and \ref{condition:main_n=2_2'} hold.
\end{proof}
The approach used in the proof of Theorem \ref{th:r=2_v=13} can be further applied for larger $v$. According to our computation, the next $10$ possible values of a prime $v$ dividing $2n^2+2n+1$ are $17,~ 29,~ 37,~ 41,~ 53,~ 61,~ 73,~ 89,~ 97,~ 101$. Our MAGMA program shows that for $v=17$, we can choose $p=3$ and follows the steps of the calculations given in Theorem \ref{th:r=2_v=13} to prove that there is no $\overline{T}$ satisfying Conditions \ref{condition:main_n=2_1'} and \ref{condition:main_n=2_2'}. As the proof is very similar, we omit it here and present the results directly. It should be pointed out that we do not need the assumption that $8n-3\neq 17k^2$ anymore as in Theorems \ref{th:r=2_v=5} and \ref{th:r=2_v=13}, because $17 \not\equiv 5 \pmod{8}$ which means the discriminant of $\Q(\sqrt{8n-3})$ is always different from $17$.
\begin{theorem}\label{th:r=2_v=17}
	Assume that $17\mid 2n^2+2n+1$ and $8n+1$ is not a square. Then there is no subset $\overline{T} \subseteq C_{17}$ of size $2n+1$ satisfying Conditions \ref{condition:main_n=2_1'} and \ref{condition:main_n=2_2'}.
\end{theorem}

Similar to Remark \ref{remark:n=5_infinite}, it is not difficult to show that Theorems \ref{th:r=2_v=13} and \ref{th:r=2_v=17} both offer us nonexistence results of $\overline{T}$ satisfying Conditions \ref{condition:main_n=2_1'} and \ref{condition:main_n=2_2'} for infinitely many $n$.

For the next possible value of $v$ which is $29$, our computer is not powerful enough to provide us the univariate polynomial with the variable $\overline{T}$ by computing the resultants of $28$ pairs of polynomials.

By Corollary \ref{coro1}, Theorems \ref{th:r=2_v=5}, \ref{th:r=2_v=13} and \ref{th:r=2_v=17}, we have the following results.
\begin{corollary}\label{coro:r=2_v=5,13,17}
	Suppose that $8n+1$ is not a square in $\Z$. Assume that one collection of the following conditions holds
	\begin{enumerate}[label=(\arabic*)]
		\item \label{item:5} $5 \mid 2n^2+2n+1$, $8n-3\neq 5k^2$ for any $k\in \Z$;
		\item \label{item:13} $13 \mid 2n^2+2n+1$, $8n-3\neq 13k^2$ for any $k\in \Z$;
		\item \label{item:17} $17 \mid 2n^2+2n+1$.
	\end{enumerate}
	Then there are no linear perfect Lee codes of radius $2$ for dimension $n$.
\end{corollary}

In Table \ref{table:r=2_v=5,13,17}, we list the cardinality of $n$ for which the existence of linear $PL(n,2)$-codes can be excluded by Corollary \ref{coro:r=2_v=5,13,17}.
\begin{table}[h!]
	\centering
	\begin{tabular}{|l|c|c|c|c|c|}
		\hline
		$N$ 		& $10$ & $10^2$ & $10^3$ & $10^4$ & $10^5$ \\
		\hline
		$\#n \leq N$ excluded by Corollary \ref{coro:r=2_v=5,13,17} \ref{item:5} & $1$ & $27$ & $356$ & $3857$ & $39537$ \\
		\hline
		$\#n \leq N$ excluded by Corollary \ref{coro:r=2_v=5,13,17} \ref{item:13}   & $0$ & $8$ & $129$ & $1458$ & $15126$ \\
		\hline
		$\#n \leq N$ excluded by Corollary \ref{coro:r=2_v=5,13,17} \ref{item:17}   & $0$ & $8$ & $108$ & $1142$ & $11659$ \\
		\hline
		$\#n \leq N$ excluded by Corollary \ref{coro:r=2_v=5,13,17} & $1$ & $38$ & $499$ & $5332$ & $54606$ \\
		\hline
	\end{tabular}
	\caption{The number of $n$ to which Corollary \ref{coro:r=2_v=5,13,17} can be applied.}\label{table:r=2_v=5,13,17}
\end{table}

\subsection{More necessary conditions}\label{sub:approach2}
As we have discussed after Theorem \ref{th:r=2_v=17}, it appears elusive to follow the same approach to prove the nonexistence results for large $n$. In the next result, we adjust our strategy. Instead of finding a complicated polynomial in term of $\overline{T}$, we consider the recursive relation modulo $p$ where $p$ divides $2n$ and try to derive necessary conditions from them.

\begin{theorem}\label{th:union_criteria}
	Let $n$ be a positive integer and $p$ be a prime divisor of $2n$. Let $G$ be an abelian group of order $2n^2+2n+1$. Suppose that $H$ is one of its subgroup of index $v$ where $v$ is prime.
	Let $f$ denote the order of $p$ modulo $v$, $l=\min \{i: p^i\equiv \pm 1 \pmod{v}\}$ and $d=(v-1)/f$. Define $m=\frac{2n^2+2n+1}{v}$, $m_1 := \min\{i : i\in \Z_{\geq 0}, i\equiv m\pmod{p} \}$, $m_2 := \min\{i : i\in \Z_{\geq 0}, i\equiv 2m\pmod{p} \}$ and
	\[\lambda = \max\{r: r\mid (p^l-1), r \mid (2^i-p^j) \text{ for } 2^i\equiv p^j \pmod{v} \text{ and } (i,j)\ne(0,0) \}.\]
	Assume that
	\begin{itemize}
		\item $2n+1$ is smaller than $m_1v$ and $m_2v$,
		\item $\sigma_{2}$ and $\sigma_p$ generates the Galois group $\Gal(\Q(\zeta_v)/\Q)$ where $\zeta_v$ is a primitive $v$-th root of unity.
	\end{itemize}
	If there is $T=\sum_{g\in G} a_g g\in \Z[G]$ with $a_g\in \{0,1\}$ satisfying Conditions \ref{condition:main_n=2_1} and \ref{condition:main_n=2_2}, then the following necessary conditions hold:
	\begin{enumerate}[label=(\roman*)]
		\item\label{item:nece_1} $\lambda\neq 1$ or $v$.
		\item\label{item:nece_2} There is $x\in \{z\in\F_{p^l}: z^{\lambda}=1\}$ such that
		\begin{equation}\label{eq:xy_cond_sum}
		m\sum_{y^{v}=1}\theta(x,y) \equiv 0 \pmod{p},	
		\end{equation}
		and
		\begin{equation}\label{eq:xy_cond_ag_bound}
		m(1-\theta(x,y)) \pmod{p} \in \{0,\dots, m\} \text{ for every $y$ satisfying }y^{v}=1,
		\end{equation}
		where
		\[ \theta(x,y)= \left\{
		\begin{array}{ll}
		\sum_{i=0}^{v-2} (xy)^{2^i}, &  \langle\sigma_{2}\rangle=\Gal(\Q(\zeta_v)/\Q),\\
		\sum_{i=0}^{d-1}\Tr_{\F_{p^f}/\F_p}((xy)^{2^i}), &  \langle\sigma_{2}, \sigma_p \rangle=\Gal(\Q(\zeta_v)/\Q).
		\end{array} \right.
		\]
		\item\label{item:nece_3} In particular, when $G=H$, Condition \ref{item:nece_2} becomes the existence of
		$x\in \{z\in\F_{p^l}: z^{\lambda}=1\}$ such that
		\begin{align}
		\label{eq: xy_cond_1} |\{y\in \F_{p^f}: \theta(x,y)=1, y^{2n^2+2n+1}=1\}| &= 2n^2,\\
		\label{eq: xy_cond_2} |\{y\in \F_{p^f}: \theta(x,y)=0, y^{2n^2+2n+1}=1\}| &= 2n+1.
		\end{align}
	\end{enumerate}
\end{theorem}
\begin{proof}
	We look at the image $\overline{T}$ of $T$ in $G/H$. By Condition \ref{condition:main_n=2_2'},
	\[\overline{T}^2 \equiv -\overline{T}^{(2)} + 2n \pmod{G/H}.  \]
	It follows that
	\[ (\overline{T}^{(2)})^2 \equiv -\overline{T}^{(4)} + 2n \pmod{G/H},\]
	which means
	\[ \overline{T}^{(2^2)} \equiv -\overline{T}^{2^2} + 2n(2\overline{T}^2-2n+1) \pmod{G/H}. \]
	After another $i-2$ steps, we get
	\begin{equation}\label{eq:T^(2^i)=-T^2^i}
	\overline{T}^{(2^i)} \equiv -\overline{T}^{2^i} + 2nL \pmod{G/H}	
	\end{equation}
	for some $L\in \Z[G/H]$.
	
	For a given non-principal character $\chi\in \widehat{G/H}$, suppose that $-\alpha = \chi(\overline{T})$. By \eqref{eq:T^(2^i)=-T^2^i},
	\begin{equation}\label{eq:T^2^i mod P}
	\chi(\overline{T}^{(2^i)})\equiv -\alpha^{2^i} \pmod{p},
	\end{equation}
	for $i=0,1,\cdots$.
	
	Since $a_{\bar{g}}^p \equiv a_{\bar{g}} \pmod{p}$,
	\begin{equation*}
	\chi(\overline{T}^{(2^ip)}) = \sum_{\bar{g}\in G/H}a_{\bar{g}} \chi(\bar{g}^{2^ip}) \equiv \left(\sum_{\bar{g}\in G/H}a_{\bar{g}} \chi(\bar{g}^{2^i})\right)^p \equiv -\chi(\overline{T}^{(2^i)}) ^p \pmod{p},
	\end{equation*}
	which with \eqref{eq:T^2^i mod P} implies
	\begin{equation}\label{eq:T_2^ip^j}
	\chi(\overline{T}^{(2^ip^j)}) \equiv -\alpha^{2^ip^j} \pmod{p},
	\end{equation}
	for $i=0,1,\cdots$ and $j=0,1,\cdots$.
	
	As $\varphi(v)=v-1=fd$ and $\langle \sigma_{2}, \sigma_p \rangle=\Gal(\Q(\zeta_v)/\Q)$,  we have
	\begin{equation}\label{eq:all_powers_exp(G)}
	\{2^jp^k \pmod{v} : 0\le j\le d-1, 0\le k\le f-1\} = \{1,\dots, v-1 \}.
	\end{equation}
	In particular, if $\langle \sigma_2 \rangle = \Gal(\Q(\zeta_v)/\Q)$, then
	\begin{equation}\label{eq:all_powers_by_2}
	\{ 2^i \pmod{v} : 0\le i \le v-2 \} = \{1,\dots, v-1 \}.
	\end{equation}
	
	Given an arbitrary non-principal $\chi_1\in \widehat{G/H}$ which obviously generates $\widehat{G/H}$, by \eqref{eq:all_powers_exp(G)} it is clear that
	\[ \{\chi_1((\cdot)^{2^jp^k}) : 0\le j\le d-1, 0\le k\le f-1\} = \widehat{G/H}\setminus \{\chi_0\}. \]
	Let us denote $\chi_1((\cdot)^{2^jp^k})$ by $\chi_{j,k}$. Assume that $\chi_{0,0}(\overline{T})\equiv -\alpha \pmod{p}$. By \eqref{eq:T_2^ip^j},
	\[ \chi_{j,k}(\overline{T})\equiv -\alpha^{2^jp^k} \pmod{p}.\]
	
	Let $\chi_1(\bar{g}^{-1}) = \beta$ which is a power of $\zeta_v$. It is clear that $\chi_{j,k}(\bar{g}^{-1})=\beta^{2^j p^k}$. By the inversion formula \eqref{eq:inversion_formula},
	\begin{align*}
	a_{\bar{g}} &= \frac{1}{|G/H|}\sum_{\chi\in \widehat{G/H}} \chi(\overline{T}) \chi(\bar{g}^{-1})\\
	&= \frac{|H|}{|G|}\left(\sum_{j,k} \chi_{j,k}(\overline{T}) \chi_{j,k}(\bar{g}^{-1})+\chi_0(\overline{T})\chi_0({\bar{g}^{-1}})\right)\\
	&\equiv m\left(\sum_{j,k} -(\alpha\beta)^{2^jp^k} + (2n+1)  \right)\pmod{p},
	\end{align*}
	whence
	\begin{equation}\label{eq:a_g_criteria1}
	a_{\bar{g}}\equiv m\left(\sum_{j,k} -(\alpha\beta)^{2^jp^k} + 1  \right)\pmod{p}.
	\end{equation}
	
	In particular, if $\langle \sigma_2 \rangle = \Gal(\Q(\zeta_v)/\Q)$, then by \eqref{eq:all_powers_by_2}
	\[ \{\chi_1((\cdot)^{2^i}) : 0\le i\le v-2\} = \widehat{G/H}\setminus \{\chi_0\}. \]
	Let us denote $\chi_1((\cdot)^{2^i})$ by $\chi_{2^i}$. Assume that $\chi_{2^0}(\overline{T})\equiv -\alpha \pmod{p}$. By \eqref{eq:T_2^ip^j},
	\[ \chi_{2^i}(\overline{T})\equiv -\alpha^{2^i} \pmod{p}.\]
	Under the assumption that $\chi_1(\bar{g}^{-1}) = \beta$ which is a power of $\zeta_v$. It is clear that $\chi_{2^i}(\bar{g}^{-1})=\beta^{2^i}$. By the inversion formula \eqref{eq:inversion_formula},
	\begin{equation}\label{eq:a_g_criteria1_by_2}
	a_{\bar{g}}\equiv m\left(\sum_{i=0}^{v-2} -(\alpha\beta)^{2^i} + 1  \right) \pmod{p}.
	\end{equation}
	
	Next let us derive some further restrictions on the value of $\alpha$ which equals $-\chi(\overline{T})$. Let $i$ and $j$ be integers such that
	\[ 2^i\equiv p^j \pmod{v},\]
	whence $\chi(\overline{T}^{(2^i)}) = \chi(\overline{T}^{(p^j)})$.
	By \eqref{eq:T_2^ip^j},
	\begin{equation}\label{eq:2^i=p^j}
	-\alpha^{2^i}\equiv \chi(\overline{T}^{(2^i)}) = \chi(\overline{T}^{(p^j)}) \equiv -\alpha^{p^j} \pmod{p}.
	\end{equation}
	
	For a prime $p$ diving $2n$, by Lemma \ref{lm:ideal_decomp}, $(p)=P_1\cdots P_d$ in $\Z[\zeta_v]$. By \eqref{eq:2^i=p^j}, we get
	\begin{equation}\label{eq:x=-x^{2^h}}
	\alpha^{2^i}\equiv\alpha^{p^j} \pmod{P_k},
	\end{equation}
	for $k=1,2,\cdots, d$. This is actually equivalent to an equation
	$x^{2^i}=x^{p^j}$
	over the residue field $\Z[\zeta_v]/P_k\cong \F_{p^f}$ where $x=\alpha+P_k$.
	
	When $f$ is even, $p^{f/2} \equiv -1 \pmod{v}$ which means $l=f/2$. Together with $\chi(\overline{T})=\chi(\overline{T}^{(-1)})$, we derive $\alpha\equiv\alpha^{p^{f/2}} \pmod{P_k}$, i.e.\ $x=\alpha+P_k$ is in the subfield $\F_{p^{f/2}}$ of $\F_{p^f}$. Of course, when $f$ is odd, $l$ still equals $f$. Thus
	\begin{equation}\label{eq:alpha=alpha^p^l}
	\alpha^{p^l}\equiv\alpha \pmod{P_k},
	\end{equation}
	for every $k$.

	Recall that we have defined
	$$\lambda = \max\{r: r\mid (p^l-1), r \mid (2^i-p^j) \text{ for } 2^i\equiv p^j \pmod{v}  \}.$$
	From \eqref{eq:x=-x^{2^h}} and \eqref{eq:alpha=alpha^p^l}, it follows that $\alpha$ must be zero or satisfy
	\begin{equation}\label{eq:alpha_mod_Pk}
	\alpha^\lambda\equiv 1\pmod{P_k}	
	\end{equation}
	for $k=1,2,\cdots, d$.	If $l=f$, then it is clear that $\lambda\geq v$; otherwise, it is possible that $\lambda<v$.
	
	\textbf{Case \RN{1}:} First let us look at the case that $\lambda \mid v$, i.e.\ $\lambda=1$ or $v$. Let $\alpha_i=\zeta_v^{\frac{iv}{\lambda}}$ for $i=0,\cdots, \lambda-1$. It is clear that $\chi(\overline{T})\equiv-\alpha_j \pmod{P_k}$ satisfies \eqref{eq:x=-x^{2^h}} for $j=0,\cdots, \lambda-1$. As there are exactly $\lambda+1$ solutions of \eqref{eq:x=-x^{2^h}} over $\F_{p^{f}}$, $\chi(\overline{T})$ must be congruent to one of the elements in $\{0, -\alpha_0, \cdots, -\alpha_{\lambda-1} \}$ modulo $P_k$.
	
	Next we show that the value of $\chi(\overline{T})$ modulo $P_k$ does not depend on $k$ under the condition that $\lambda \mid v$.	Suppose that $\chi(\overline{T})\equiv-\alpha \pmod{P_1}$ where $\alpha$ is a power of $\zeta_v$. Let $h$ be the order of $2$ modulo $v$. Then
	$-\alpha\equiv-\alpha^{2^h}\pmod{P_1}$. By \eqref{eq:T^2^i mod P},
	\begin{equation}\label{eq:T^2^i mod P_1}
	\chi(\overline{T}^{(2^i)}) \equiv -\alpha^{2^i} \pmod{P_1},
	\end{equation}
	for $i=0,1,\dots$. Applying the field automorphism $\sigma_2^{h-i}$ on it, we obtain
	\[ \chi(\overline{T}) \equiv -\alpha^{2^h}\equiv -\alpha \pmod{\sigma_2^{h-i}(P_1)}.\]
	By the assumption that $\sigma_2$ acts transitively on $P_k$'s,
	\[ \chi(\overline{T}) \equiv -\alpha \pmod{P_k},\]
	for $k=1,2,\dots, d$. Thus $\chi(\overline{T}) \equiv -\alpha \pmod{p}$.
	
	
	Recall that all the possible values of $\alpha$ are $0$ and all the $\lambda$-th roots of unity $\alpha_i$ in $\Z[\zeta_v]$.
	
	If $\alpha=0$, then \eqref{eq:a_g_criteria1} shows us that $a_{\bar{g}}\equiv m \pmod{p}$ for all ${\bar{g}}$. As every $a_{\bar{g}}\ge 0$, $a_{\bar{g}}\geq m_1$. By assumption, $vm_1>2n+1$, it follows that $\sum_{g\in G}a_{\bar{g}}>2n+1$ which contradicts $|T|=2n+1$.
	
	If $\alpha$ is nonzero, then it is a power of $\zeta_v$ and $\alpha\beta$ is again a power of $\zeta_v$. By \eqref{eq:a_g_criteria1} and \eqref{eq:all_powers_exp(G)},
	\[a_{\bar{g}} \equiv m\left(-\sum_{i=1}^{v-1} (\alpha\beta)^{i}+1\right)\equiv 2m\equiv m_2\pmod{p},\]
	for all $\bar{g}$. It follows that $a_{\bar{g}} \ge m_2$ whence $2n+1 \ge v m_2$. This contradicts the assumption that $2n+1<vm_2$.
	
	\textbf{Case \RN{2}:} If $\lambda\neq 1$ or $v$ which implies that $\alpha=-\chi(\overline{T})$ can take a value that is neither zero nor a power of $\zeta_v$, we consider \eqref{eq:T^2^i mod P} modulo $P$ where $P$ is one of the prime ideals containing $(p)$.
	
	By \eqref{eq:a_g_criteria1},
	\begin{equation*}
	a_{\bar{g}}\equiv m\left(\sum_{j,k} -(\alpha\beta)^{2^jp^k} + 1\right)\pmod{P}.
	\end{equation*}
	
	Now it can be regarded as an equation over the finite field $\Z[\zeta_v]/P \cong \F_{p^f}$. Let us use $x$ and $y\in \F_{p^f}$ to denote $\alpha\pmod{P}$ and $\beta\pmod{P}$. Then we have
	\begin{equation*}
	a_{\bar{g}}\equiv m\left(1
	-\sum_{j=0}^{d-1} \Tr_{\F_{p^f}/\F_p}((xy)^{2^j})\right) \pmod{p}.
	\end{equation*}
	
	In the special case that $\langle \sigma_2 \rangle = \Gal(\Q(\zeta_v)/\Q)$. By \eqref{eq:a_g_criteria1_by_2},
	\begin{align*}
	a_{\bar{g}}&\equiv \left(\sum_{i=0}^{v-2} -(\alpha\beta)^{2^i} + 1\right)  \pmod{P}
	\end{align*}
	which can be viewed as an equation in the finite field $\F_{p^f}$
	\[ a_{\bar{g}}\equiv m\left( 1-\sum_{i=0}^{v-2} (xy)^{2^i}\right) \pmod{p}.\]
	
	If $x=0$, then $a_{\bar{g}}\equiv m\pmod{p}$ for every $\bar{g}$, whence $2n+1=\sum a_{\bar{g}} \geq vm_1$. It contradicts our assumption on the value of $vm_1$. Thus $x$ must be nonzero. By \eqref{eq:alpha_mod_Pk}, $x^\lambda=1$.
	
	As  $a_{\bar{g}}$ is an integer between $0$ and $m$, we obtain \eqref{eq:xy_cond_ag_bound}.	By calculation, $mv\equiv 1\pmod{p}$ and
	\begin{align*}
	\sum a_{\bar{g}} &= m\left(v- \sum_{y^{v}=1}\theta(x,y)\right) \pmod{p}\\
	&= 1-m\left(\sum_{y^{v}=1}\theta(x,y)\right) \pmod{p}.
	\end{align*}
	Since $\sum a_{\bar{g}}  = 2n+1\equiv 1\pmod{p}$, we obtain \eqref{eq:xy_cond_sum}.
	
	In the special case $H=\{0\}$, i.e.\ $m=1$ and $v=2n^2+2n+1$. Noticing that $a_g\equiv 0$ or $1 \pmod{P}$, we obtain the necessary conditions \eqref{eq: xy_cond_1} and \eqref{eq: xy_cond_2}.
\end{proof}

\begin{example}
	We can use Theorem \ref{th:union_criteria} to show the following nonexistence results.
	\begin{itemize}
		\item When $n=102$, $|G|=21013$ is a prime and $h=10506$. Take $p=3$, we have $l=5253$, $f=2l$ and $\gcd(2^h-1, p^l-1)=1$ which means $\lambda=1$. As $\varphi(|G|)=2f$, $d=2$ and $\sigma_2$ acts transitively on $P_1$ and $P_2$. Thus, by Theorem \ref{th:union_criteria}, there is no $T$ satisfying Conditions \ref{condition:main_n=2_1} and \ref{condition:main_n=2_2} for $n=102$.
		\item When $n=14$, $|G|=421$ is a prime henceforth $v$ can only take the value $421$. By computation, $\lambda=3\neq 1$ or $v$. Thus we cannot use the first necessary condition to exclude the existence of $T$. However, let $p=7$ from which it follows that $l=f/2=35$. By our MAGMA program, there does not exist $x\in \F_{7^l}$ satisfying \eqref{eq:xy_cond_sum} and \eqref{eq:xy_cond_ag_bound} simultaneously. Moreover, \eqref{eq: xy_cond_1} and \eqref{eq: xy_cond_2} also do not hold.
	\end{itemize}
\end{example}

\begin{remark}
	In Theorem \ref{th:union_criteria}, we allow $p=2$. However, it is worth pointing out that when $p=2$ we cannot derive any useful information from Theorem \ref{th:union_criteria}, because in that case we are essentially using $T^2 \equiv T^{(2)} \pmod{2}$ which holds for any $T\in \Z[G]$.
\end{remark}
\begin{remark}
	If $\langle\sigma_{2}, \sigma_p,\sigma_{-1} \rangle=\Gal(\Q(\zeta_v)/\Q)$, we can still derive $\chi(\overline{T})$ for every character $\chi\in \widehat{G/H}$ from one value $\chi_1(\overline{T})$ by using \eqref{eq:T_2^ip^j} and
	\begin{equation}\label{eq:T_-2^ip^j}
	\chi(\overline{T}^{(-2^ip^j)}) \equiv \chi(\overline{T}^{(2^ip^j)})\equiv  -\alpha^{2^ip^j} \pmod{p},
	\end{equation}
	for $i=0,1,\cdots$ and $j=0,1,\cdots$, which follows from Condition \ref{condition:main_n=2_1'}.
	
	However, this case actually coincides with the one that $\langle\sigma_{2}, \sigma_p\rangle=\Gal(\Q(\zeta_v)/\Q)$, i.e.\ it is never possible that $\langle\sigma_{2}, \sigma_p \rangle\neq \langle\sigma_{2}, \sigma_p,\sigma_{-1} \rangle=\Gal(\Q(\zeta_v)/\Q)$. Assume, by way of contradiction, that this case happens. A necessary condition is that $v-1 \not\equiv 0 \pmod{4}$. As $v$ is a prime dividing $2n^2+2n+1$, the following equation
	\[ 2x^2+2x+1 =0 \pmod{v} \]
	has solutions. This implies that $-1$ must be a square in $\F_v$. Hence $v\equiv 1\pmod{4}$ which contradicts the necessary condition.
\end{remark}

By Corollary \ref{coro1} and Theorem \ref{th:union_criteria}, we have the following result.
\begin{corollary}\label{coro:union_criteria}
	Assume that at least one of the necessary conditions \ref{item:nece_1}, \ref{item:nece_2} and \ref{item:nece_3} in Theorem \ref{th:union_criteria} is not satisfied for at least one of the prime divisor $v$ of $2n^2+2n+1$. Then there are no linear perfect Lee codes of radius $2$ for dimension $n$.
\end{corollary}

For large $n$, it becomes more difficult to check the necessary conditions \ref{item:nece_2} and \ref{item:nece_3} in Corollary \ref{coro:union_criteria}. The main reason is that the size of the finite field $\F_{p^l}$ is increasing very fast as $n$ is getting larger. On our computer, we can use MAGMA to test these necessary conditions for $n$ up to $50$. In contrast, it is much easier to check the necessary condition \ref{item:nece_1} on $\lambda$.

In the Appendix, we list all the nonexistence results of linear $PL(n,2)$-codes for $3\leq n \leq 100$ in Table \ref{table:100}. There are totally $10$ integers $n=3$, $10$, $16$, $21$, $36$, $55$, $64$, $66$, $78$, $92$ that we cannot disprove the existence of $PL(n,2)$-codes using Theorem \ref{th:kim_generalization}, Corollaries \ref{coro:r=2_v=5,13,17} and \ref{coro:union_criteria}. However, by \cite[Theorem 6]{H09E} and \cite[Theorem 8]{HG14}, we can further exclude the case with $n=3$ and $10$, respectively. Therefore, the existence of $PL(n,2)$-codes are still open for exactly $8$ integers $n$ for $n\leq 100$.

It is worth pointing out that each of Theorem \ref{th:kim_generalization}, Corollary \ref{coro:r=2_v=5,13,17} and Corollary \ref{coro:union_criteria} offers us the nonexistence of linear $PL(n,2)$-codes which cannot be covered by the other two criteria. For example, only Theorem \ref{th:kim_generalization} works for $n=19$ and only Corollary \ref{coro:union_criteria} works for $n=20$. It is clear that Corollary \ref{coro:r=2_v=5,13,17} is a summary of Theorems \ref{th:r=2_v=5}, \ref{th:r=2_v=13} and \ref{th:r=2_v=17}. Each of these three Theorems provides us irreplaceable nonexistence results of linear $PL(n,2)$-codes. For instance, $n=8$ is only covered by Theorem \ref{th:r=2_v=5}, $n=49$ is only covered by Theorem \ref{th:r=2_v=13} and $n=299$ is only covered by Theorem \ref{th:r=2_v=17}.

In Table \ref{table:comparing}, we compare the cardinality of integers $n$ for which the nonexistence of linear $PL(n,2)$-codes are proved in this paper and those in \cite{K17} and \cite{QCC18}.
\begin{table}[h]
	\centering
	\begin{tabular}{|c|c|c|c|c|c|c|}
		\hline
		$N$ &  \cite{K17} & \cite{QCC18} & Theorem \ref{th:kim_generalization} & Corollary \ref{coro:r=2_v=5,13,17} & Corollary \ref{coro:union_criteria} \ref{item:nece_1} & All\\
		\hline
		$100$ & $34$ & $16$ &  $68$ & $38$ & $44$ & $90$ \\
		\hline
		$500$ & $129$ & $103$ & $360$ & $236$& $247$ & $462$ \\
		\hline
	\end{tabular}
	\caption{The cardinality of $n\leq N$ for which linear $PL(n,2)$-codes is proved to be not exist.}\label{table:comparing}
\end{table}

\section{Radius equals 3}\label{radius3}
In this part, we investigate the existence of perfect Lee codes with radius $r=3$. For a positive integer $n$, let $G$ be a multiplicative group of order $\frac{4n(n-1)(n-2)}{3}+6n^2+1$ with identity element denoted by $1$. By Lemma~\ref{lemma2}, we consider the existence of $T\subseteq G$ viewed as an element in $\Z[G]$ satisfying
\begin{enumerate}[label=(\roman*)]
	\item \label{condition:main_n=3_0} $1\in T$,
	\item \label{condition:main_n=3_1} $T=T^{(-1)}$,
	\item \label{condition:main_n=3_2} $T^3= 6G-3T^{(2)}T -2T^{(3)}+6nT$.
\end{enumerate}

Let $H$ be a subgroup of $G$. Suppose that $m=|H|$ and $v=|G/H|$. As in the proofs of Theorems \ref{th:r=2_v=5}, \ref{th:r=2_v=13}, \ref{th:r=2_v=17} and \ref{th:union_criteria}, we define $\overline{T}$ as the image of $T$ under the canonical homomorphism from $G$ to $G/H$. Now Conditions \ref{condition:main_n=3_1} and \ref{condition:main_n=3_2} becomes
\begin{enumerate}[label=(\roman*')]\setcounter{enumi}{1}
	\item \label{condition:main_n=3_1'} $\overline{T}=\overline{T}^{(-1)}$,
	\item \label{condition:main_n=3_2'} $\overline{T}^3= 6m G/H-3\overline{T}^{(2)}\overline{T} -2\overline{T}^{(3)}+6n\overline{T}$.
\end{enumerate}

One possible solution to Condition \ref{condition:main_n=3_2'} is that $\overline{T}=\overline{T}^{(2)}=\overline{T}^{(3)}$. In fact, if $v\mid (2n+1)$, then $\overline{T} = \frac{2n+1}{v}G/H$ satisfies Condition \ref{condition:main_n=3_2'}. In such a case, we cannot use it to prove the nonexistence of $T$ satisfying Conditions \ref{condition:main_n=3_1} and \ref{condition:main_n=3_2}. Hence, to obtain nonexistence results, we have to exclude several special parameters for $\overline{T}$.
\begin{lemma}\label{lm:SS}
	Let $S=a+bG\in \Z[G]$ with $|G|=v$. Assume that positive integers $v$ and $m$ are such that  $a+vb=2n+1$, $mv=1+6n^2+\frac{4n(n-1)(n-2)}{3}$ and $S$ satisfies
	\begin{equation}\label{eq:SS}
		S^3 = 6m G-3S^2-2S+6nS.
	\end{equation}
	Then one of the following conditions holds.
	\begin{itemize}
		\item $v \mid 2n+1$;
		\item $24n+1$ is a square $c^2\in \Z$ and $v$ divides at least one of $\frac{c^2\pm 6c+29}{12}$.
	\end{itemize}
\end{lemma}
\begin{proof}
	It can be readily verified that
	\[ S^2 = ((2n+1)b+ab)G + a^2, \]
	and
	\[ S^3 = b((2n+1+a)(2n+1)+a^2)G + a^3.  \]
	By comparing the coefficients of $1$ and $G$ in \eqref{eq:SS}, we have
	\begin{align}
		\label{eq:SS_1} f_1&=a^3 + 3a^2+2a-6an=0 \\
		\label{eq:SS_2} f_2&=b((2n+1+a)(2n+1)+a^2 + 5+3a )-6m=0.
	\end{align}
	Obviously, \eqref{eq:SS_1} holds if and only if $a=0$ or there exists an integer $a$ such that $a^2+3a+2-6n=0$. It is clear that $a=0$ means $v\mid 2n+1$. The second case is equivalent to that $1+24n$ is a square in $\Z$. In such a case,
	\[bv = 2n+1-a = \frac{c^2-1}{12}+1 - \frac{-3\pm c}{2}=\frac{c^2\mp 6c+29}{12}. \qedhere\]
\end{proof}
\begin{remark}
	In fact, \eqref{eq:SS_2} in Lemma \ref{lm:SS} cannot provide us more restrictions than \eqref{eq:SS_1}.	Recall that $2n+1=a+bv$ and $3+18n^2+4n(n-1)(n-2)=3mv$. By plugging them into \eqref{eq:SS_2}, we get \eqref{eq:SS_1} multiplying a constant.
\end{remark}

One difficulty to get a result here similar to Theorem \ref{th:union_criteria} is that it is not easy to derive the value of $\chi(\overline{T}^{i})$ from $\chi(\overline{T})$. However, for small $v$, it is still possible to obtain some results as what we have done in Theorems \ref{th:r=2_v=5}, \ref{th:r=2_v=13} and \ref{th:r=2_v=17}. The first three smallest values of $v$ dividing $|G|$ for some $n$ are $v=3,5,7$.

When $v=3$, by Condition \ref{condition:main_n=3_1'}, $\overline{T}^{(2)} =\overline{T}^{(-1)}=\overline{T}$ and $\overline{T}^{(3)} = 2n+1$. Hence Condition \ref{condition:main_n=3_2'} becomes
\begin{equation}\label{eq:v=3-1}
\overline{T}^3= 6m G/H-3\overline{T}^{2}-2(2n+1)+6n\overline{T}.
\end{equation}
Let $G/H = \{1, g, g^{-1}\}$. Assume that $\overline{T}= a + b(g+g^{-1})$ for some $a,b\in \Z$. By calculation, one can see that there always exist $a$ and $b$ satisfying \eqref{eq:v=3-1}.

When $v=5$, it is direct to verify that $\overline{T}=bG/H$ for the integer $b=(2n+1)/5$ satisfying Conditions \ref{condition:main_n=3_1'} and \ref{condition:main_n=3_2'}, i.e.\ it cannot be excluded by Lemma \ref{lm:SS}.

Fortunately $v=7$ is not the trivial case anymore and we can completely prove the following nonexistence results.
\begin{theorem}\label{th:r=3_v=7}
	Assume that $n\equiv 1,5 \pmod{7}$. If $24n+1$ is not a square or $84 \nmid (24n+1)^2\pm 6\sqrt{24n+1}+29$, then there is no $\overline{T}$ of size $2n+1$ in $C_7$ satisfying Conditions \ref{condition:main_n=3_1'} and \ref{condition:main_n=3_2'}.
\end{theorem}
\begin{proof}
Now $v=7$. It is straightforward to verify that, when $n\equiv 1,5 \pmod{7}$, $7$ divides $|G|$. This also means that $7\nmid 2n+1$. Thus the first necessary condition in Lemma \ref{lm:SS} never holds.

By Condition \ref{condition:main_n=3_1'}, $\overline{T}^{(2^2)} = \overline{T}^{(-3)} = \overline{T}^{(3)}$ and $\overline{T}^{(2^3)}=\overline{T}$. Now Condition \ref{condition:main_n=3_2'} becomes
\begin{equation}\label{eq:v=7-1}
	f_1=\overline{T}^3-6m G/H+3\overline{T}^{(2)}\overline{T}+2\overline{T}^{(2^2)}-6n\overline{T}=0.
\end{equation}
Replacing $\overline{T}$ by $\overline{T}^{(2)}$ and $\overline{T}^{(2^2)}$ in \eqref{eq:v=7-1} respectively, together with Condition \ref{condition:main_n=3_1'} we get
\begin{equation}\label{eq:v=7-2}
	f_2 = (\overline{T}^{(2)})^3 - 6m G/H + 3\overline{T}^{(2^2)}\overline{T}^{(2)} + 2\overline{T} - 6n\overline{T}^{(2)}=0,
\end{equation}
and
\begin{equation}\label{eq:v=7-3}
	f_3 = (\overline{T}^{(2^2)})^3 - 6mG/H + 3\overline{T}\overline{T}^{(2^2)}  + 2\overline{T}^{(2)} - 6n\overline{T}^{(2^2)}=0.
\end{equation}
These three equations can be regarded as polynomials with variables $\overline{T}$, $\overline{T}^{(2)}$ and $\overline{T}^{(3)}$.
By using MAGMA~\cite{Magma}, we calculate the resultant $f_{12}$ of $f_1$ and $f_2$ as well as the resultant $f_{13}$ of $f_1$ and $f_3$ with respect to $\overline{T}^{(2)}$. Then we compute the resultant $g$ of $f_{12}$ and $f_{13}$ with respect to $\overline{T}^{(2^2)}$. This means $g$ only contains variables $\overline{T}$ and $G/H$. Moreover $g$  modulo $G/H$ can be factorized into the multiplication of $3$ irreducible ones
\begin{equation}\label{eq:v=7_h}
	h\equiv 3\overline{T} (\overline{T}^2 + 3\overline{T} - 6n + 2)h_3 \pmod{G/H},
\end{equation}
where $h_3$ is of degree $24$. As $h_3$ is very complicated, we do not write them down here explicitly. Since $h$ is derived from \eqref{eq:v=7-1}, \eqref{eq:v=7-2} and \eqref{eq:v=7-3}, it must be congruent to $0$ modulo $G/H$.

Our assumptions imply that the necessary conditions in Lemma \ref{lm:SS} are not satisfied, which means that $\overline{T} (\overline{T}^2 + 3\overline{T} - 6n + 2)\neq 6m G/H$. This and a direct counting argument further imply that
\[\overline{T} (\overline{T}^2 + 3\overline{T} - 6n + 2)\not\equiv 0 \pmod{G/H}.\]
As $|G/H|=7$ is a prime, there is no zero divisors in $\Z[G/H]$. Hence, $h\equiv 0\pmod{G/H}$ implies
\[h_3\equiv 0\pmod{G/H}.\]
By the symmetric property of $\overline{T}$, $\overline{T}^{(2)}$ and $\overline{T}^{(3)}$, it also holds if we replace $\overline{T}$ in it by $\overline{T}^{(2)}$ or $\overline{T}^{(3)}$.

Let $\chi\in \widehat{G/H}$ be a non-principal character. It follows that
\[ \chi(h_3) =0. \]
As $h_3$ is too complicated, we cannot handle them directly. Instead, we choose some prime number $p$ and consider $\chi(h_3)$ modulo $p$. If $p$ is primitive modulo $v$ which means $p\equiv 3,5 \pmod{7}$, by Lemma \ref{lm:ideal_decomp}, $(p)$ is still a prime ideal in $\Z[\zeta_v]$. If we replace $\chi(\overline{T}) \pmod{p}$ by $X$ in $\chi(h_3)\equiv 0\pmod{p}$ and let its coefficients modulo $p$, then we get a polynomial $\bar{h}_3(X)$ in $\F_{p}[X]$. Let $\tau_1 \equiv \chi(\overline{T})\pmod{p}$, it is clear that  $\bar{h}_3(\tau_1)=0$.

Note that $\tau_1$ is a root of $\bar{h}_3$.
As $\overline{T}^{(-1)}=\overline{T}$,  $\chi(\overline{T}^{(-1)})\equiv \chi(\overline{T}) \pmod{p}$ whence $\tau_1$ is in $\F_{p^{(v-1)/2}} = \F_{p^3}$. Recall that $\chi(\overline{T}^{(p)})\equiv \chi(\overline{T})^p \pmod{p}$. Hence
\begin{equation}\label{eq:tau_2}
	\chi(\overline{T}^{(2)})
	\equiv \tau_2= \left\{
	\begin{array}{ll}
	\tau_1^p \pmod{p},     & p\equiv 5 \pmod{7}, \\
	\tau_1^{p^2} \pmod{p}, & p\equiv 3 \pmod{7}.
	\end{array}
	\right.
\end{equation}
and
\begin{equation}\label{eq:tau_3}
\chi(\overline{T}^{(3)})
\equiv \tau_3=\left\{
\begin{array}{ll}
\tau_1^p \pmod{p},    & p\equiv 3 \pmod{7}, \\
\tau_1^{p^2} \pmod{p}, & p\equiv 5 \pmod{7}.
\end{array}
\right.
\end{equation}

Let us consider the necessary conditions that $\tau_1$ must satisfy. First, by \eqref{eq:v=7-1}, \eqref{eq:v=7-2} and \eqref{eq:v=7-3},
\begin{align}
	\label{eq:v=7_tau_1} \tau_1^3+3\tau_2\tau_1+2\tau_3-6n\tau_1&=0 \\
	\label{eq:v=7_tau_2} \tau_2^3+3\tau_3\tau_2+2\tau_1-6n\tau_2&=0 \\
	\label{eq:v=7_tau_3} \tau_3^3+3\tau_1\tau_3+2\tau_2-6n\tau_3&=0.
\end{align}

Second, as in Theorem \ref{th:union_criteria}, we look at the coefficients of $a_{\bar{g}}$ by using the inversion formula \eqref{eq:inversion_formula}. Let $\beta$ be an element of order $7$ in $\F_{p^6}$. For $\bar{g} \in G/H$ with $\chi(\bar{g}) \equiv \beta \pmod{p}$,
\begin{align}
	\nonumber	a_{\bar{g}} &= \frac{1}{7} \left((2n+1)  + \sum_{i=1}^{6}\chi(\overline{T}^{(i)})\chi(\bar{g}^{-i}) \right)\\
	\nonumber &=\frac{1}{7} \left((2n+1)  + \sum_{i=1}^{3} \chi(\overline{T}^{(i)})(\chi(\bar{g}^{i}) + \chi(\bar{g}^{-i})) \right)\\
	\label{eq:v=7_ag=...}&\equiv
	\dfrac{1}{7} \left( 2n+1 + \tau(\beta+\beta^{-1}) + \tau_2(\beta^2+\beta^{-2}) + \tau_3(\beta^3+\beta^{-3}) \right) \pmod{p}.
\end{align}
As $|\overline{T}| = |T|=2n+1$, every $a_{\bar{g}}$ from \eqref{eq:v=7_ag=...} should also satisfy that
\begin{equation}\label{eq:v=7_ag_sum}
	\sum_{\bar{g}\in G/H} a_{\bar{g}} \equiv 2n+1\pmod{p}.
\end{equation}

Let us take $p=5$. In MAGMA, the defining polynomial for $\F_{5^3}$ is $X^3+3X+3\in \F_{5}[X]$. Let $\xi$ denote a root of it. Depending on the value of $n$ modulo $p$, we divide our proof into the following cases.

\textbf{Case \RN{1}:} Assume that $n\equiv 4\pmod{5}$. By using MAGMA, we can show that $\bar{h}_3$ has no root in $\F_{p^{3}}$, which contradicts $h_3\equiv 0 \pmod{G/H}$.

\textbf{Case \RN{2}:} Assume that $n\equiv 2\pmod{5}$. By using MAGMA, it can be checked that $\bar{h}_3$ only has roots $\tau_1= 3,4$. By \eqref{eq:tau_2} and \eqref{eq:tau_3}, $\tau_2=\tau_3=\tau_1$ which implies that \eqref{eq:v=7_tau_1}, \eqref{eq:v=7_tau_2} and \eqref{eq:v=7_tau_3} are the same equation. Plugging $\tau_1=3$ or $4$ into it, it is not equal to $0$ which is a contradiction.

\textbf{Case \RN{3}:} Assume that $n\equiv 0\pmod{5}$. By MAGMA, the roots of $\bar{h}_3$ are  $\tau_1=\xi^{42}, \xi^{58}, \xi^{86}$. It is straightforward to check that \eqref{eq:v=7_tau_1}, \eqref{eq:v=7_tau_2} and \eqref{eq:v=7_tau_3} hold. However, the left-hand-side of \eqref{eq:v=7_ag_sum} always equals $4$ which is different from $2n+1\equiv 1 \pmod{p}$.

\textbf{Case \RN{4}:} Assume that $n\equiv 3\pmod{5}$. By MAGMA, all the roots of $\bar{h}_3$ are
\[ \xi^{18}, \xi^{21}, \xi^{29}, \xi^{78}, \xi^{90}, \xi^{105}. \]
For $\tau_1=\xi^{18}, \xi^{78}, \xi^{90}$, it is straightforward to check that \eqref{eq:v=7_tau_1}, \eqref{eq:v=7_tau_2} and \eqref{eq:v=7_tau_3} cannot hold simultaneously. For $\tau_1=\xi^{21}, \xi^{29}, \xi^{105}$, the left-hand-side of \eqref{eq:v=7_ag_sum} always equals $3\not\equiv 2n+1\equiv 2 \pmod{p}$.

\textbf{Case \RN{5}:} Assume that $n\equiv 1\pmod{5}$. By MAGMA, $$\tau_1=1, \xi^{22}, \xi^{38}, \xi^{54}, 4, \xi^{66}, \xi^{73}, \xi^{82}, \xi^{89}, \xi^{110}, \xi^{117}.$$
It can be directly verified that $\sum a_{\bar{g}}\not\equiv 2n+1\equiv 3\pmod{p}$ for each of them.

To summary, under the two assumptions that
\begin{itemize}
	\item $7\mid v$;
	\item the necessary condition in Lemma \ref{lm:SS} does not hold;
\end{itemize}
we have proved that \eqref{eq:v=7-1}, \eqref{eq:v=7-2} and \eqref{eq:v=7-3} cannot hold simultaneously. Therefore, there is no $\overline{T}$ satisfying Conditions \ref{condition:main_n=3_1'} and \ref{condition:main_n=3_2'},
\end{proof}
\begin{remark}
	When $n\equiv 5\pmod{7}$, $24n+1$ can never be a square. It follows that Theorem \ref{th:r=3_v=7} provides us a nonexistence results of $T$ satisfying Conditions \ref{condition:main_n=3_1} and \ref{condition:main_n=3_2} for infinitely many $n$.
\end{remark}

By Corollary \ref{coro1} and Theorem \ref{th:r=3_v=7}, we have the following result.
\begin{corollary}\label{coro:r=3_v=7}
	Assume that $n\equiv 1,5 \pmod{7}$. If $24n+1$ is not a square or $84 \nmid (24n+1)^2\pm 6\sqrt{24n+1}+29$, then there are no linear perfect Lee codes of radius $3$ for dimension $n$.
\end{corollary}

In Table \ref{table:r=3_v=7}, we list the cardinality of $n$ for which the existence of linear $PL(n,3)$-codes can be excluded by Corollary \ref{coro:r=3_v=7}. Compared with Table 1 in \cite{ZG17}, here we can exclude the existence of linear $PL(n,3)$-codes for more $n$.
\begin{table}[h!]
	\centering
	\begin{tabular}{|c|c|c|c|c|c|}
		\hline
		$N$ 		& $10$ & $10^2$ & $10^3$ & $10^4$ & $10^5$ \\
		\hline
		$\# \{n\leq N: \text{Corollary \ref{coro:r=3_v=7} can be applied}\}$  & $1$ & $20$ & $256$ & $2763$ & $28267$ \\
		\hline
	\end{tabular}
	\caption{The number of $n$ to which Corollary \ref{coro:r=3_v=7} can be applied.}\label{table:r=3_v=7}
\end{table}

It appears that our approach can be further applied to analyze the existence of linear $PL(n,3)$-codes with $v=11,13,17\dots$.
First, we have to exclude the possibility that $v\mid 2n+1$, otherwise by Lemma \ref{lm:SS} we can take $\overline{T}=\frac{2n+1}{v} G/H$. It is routine to verify that there exists $v$ such that $v\mid |G|$ and $v\nmid 2n+1$ if and only if $v\equiv 7\pmod{8}$. Hence, the next interesting case is $v=23$. However, it is already beyond our computer capability to eliminate variables to obtain a univariate polynomial analogous to \eqref{eq:v=7_h}.

\section{Concluding remarks}\label{sec:conclusion}
In this paper, we have proved several nonexistence results concerning linear $PL(n,r)$-codes for $r=2,3$ via the group ring approach. Theorem \ref{th:kim_generalization}, as a generalization of Kim's result in \cite{K17}, together with Corollaries \ref{coro:r=2_v=5,13,17} and \ref{coro:union_criteria} provide us strong criteria to exclude the existence of linear $PL(n,2)$-codes for infinitely many values of $n$. In particular, together with some known results for $3\le n\le 12$ in \cite{H09E} and \cite{HG14}, we show that there is no linear $PL(n,2)$-codes in $\Z^n$ for all $3\le n\le 100$ except 8 cases which is summarized in Table \ref{table:100}. In Section \ref{radius3}, we follow the same approach to prove the nonexistence of linear $PL(n,3)$-codes for infinitely many values of $n$.

It appears that our approach can be further applied on the nonexistence problem of $PL(n,r)$ for $r>3$. However, for $r=4,5,\cdots$, the group ring equations become more involved and contain much more terms. For instance, when $r=4$, according to our computation, the corresponding group ring condition becomes
\[ T^4 = 24G -12n(T^2 +T^{(2)}) - 6T^{(2)} T^2 - 3 T^{(2)}T^{(2)} - 8 T^{(3)}T - 6T^{(4)} + 12n(n-1).\]
It seems almost infeasible to consider its image under $\rho: G\rightarrow G/H$ to derive a univariable polynomial in $\overline{T}$ which is analogous to \eqref{eq:v=7_h}.

\section*{Acknowledgment}
The authors express their gratitude to the anonymous reviewers for their detailed and constructive comments which are very helpful to the improvement of the presentation of this paper. Tao Zhang is partially supported by the National Natural Science Foundation of China (No.\ 11801109). Yue Zhou is partially supported by the National Natural Science Foundation of China (No.\ 11771451).

\section*{Appendix}
\begin{longtable}{|c|l|c|l|}
	\caption{Nonexistence of linear $PL(n,2)$-codes for dimension $n$, where $3\leq n\leq 100$.}\label{table:100}\\
	\hline \multicolumn{1}{|c|}{$n$} & \multicolumn{1}{c|}{\textbf{nonexistence}} & \multicolumn{1}{c|}{$n$} & \multicolumn{1}{c|}{\textbf{nonexistence}} \\ \hline
	\endfirsthead
	
	\multicolumn{4}{c}%
	{{\bfseries \tablename\ \thetable{} -- continued from previous page}} \\
	\hline \multicolumn{1}{|c|}{$n$} & \multicolumn{1}{c|}{\textbf{nonexistence}} & \multicolumn{1}{c|}{$n$} & \multicolumn{1}{c|}{\textbf{nonexistence}} \\ \hline
	\endhead
	
	\hline \multicolumn{4}{|r|}{{Continued on next page}} \\ \hline
	\endfoot
	
	\hline \hline
	\endlastfoot
	
	3 &\cite[Theorem 6]{H09E}&
	4 &Theorem \ref{th:kim_generalization} \\ \hline
	5 &Theorem \ref{th:kim_generalization}, Corollary \ref{coro:union_criteria}&
	6 &Theorem \ref{th:kim_generalization}, Corollary \ref{coro:union_criteria}\\ \hline
	7 &Theorem \ref{th:kim_generalization} &
	8 &Corollary \ref{coro:r=2_v=5,13,17}\\ \hline
	9 &Theorem \ref{th:kim_generalization}, Corollary \ref{coro:union_criteria}&
	10 & \cite[Theorem 8]{HG14}\\ \hline
	11 &Theorem \ref{th:kim_generalization}, Corollaries \ref{coro:r=2_v=5,13,17}, \ref{coro:union_criteria}&
	12 &Theorem \ref{th:kim_generalization} \\ \hline
	13 &Theorem \ref{th:kim_generalization}, Corollaries \ref{coro:r=2_v=5,13,17}, \ref{coro:union_criteria}&
	14 &Theorem \ref{th:kim_generalization} \\ \hline
	15 &Corollary \ref{coro:union_criteria}&
	16 &?\\ \hline
	17 &Theorem \ref{th:kim_generalization}, Corollary \ref{coro:union_criteria}&
	18 &Corollaries \ref{coro:r=2_v=5,13,17}, \ref{coro:union_criteria}\\ \hline
	19 &Theorem \ref{th:kim_generalization} &
	20 &Corollary \ref{coro:union_criteria}\\ \hline
	21 &?&
	22 &Theorem \ref{th:kim_generalization} \\ \hline
	23 &Corollaries \ref{coro:r=2_v=5,13,17}, \ref{coro:union_criteria}&
	24 &Theorem \ref{th:kim_generalization} \\ \hline
	25 &Theorem \ref{th:kim_generalization}, Corollary \ref{coro:union_criteria}&
	26 &Theorem \ref{th:kim_generalization}, Corollaries \ref{coro:r=2_v=5,13,17}, \ref{coro:union_criteria}\\ \hline
	27 &Theorem \ref{th:kim_generalization}, Corollaries \ref{coro:r=2_v=5,13,17}, \ref{coro:union_criteria}&
	28 &Corollary \ref{coro:union_criteria}\\ \hline
	29 &Theorem \ref{th:kim_generalization}, Corollary \ref{coro:union_criteria}&
	30 &Theorem \ref{th:kim_generalization}, Corollary \ref{coro:union_criteria}\\ \hline
	31 &Theorem \ref{th:kim_generalization} &
	32 &Theorem \ref{th:kim_generalization} \\ \hline
	33 &Theorem \ref{th:kim_generalization}, Corollary \ref{coro:r=2_v=5,13,17}&
	34 &Theorem \ref{th:kim_generalization} \\ \hline
	35 &Theorem \ref{th:kim_generalization} &
	36 &?\\ \hline
	37 &Theorem \ref{th:kim_generalization}, Corollary \ref{coro:union_criteria}&
	38 &Corollary \ref{coro:r=2_v=5,13,17}\\ \hline
	39 &Theorem \ref{th:kim_generalization} &
	40 &Theorem \ref{th:kim_generalization}, Corollary \ref{coro:r=2_v=5,13,17}\\ \hline
	41 &Corollary \ref{coro:r=2_v=5,13,17}&
	42 &Theorem \ref{th:kim_generalization}, Corollary \ref{coro:union_criteria}\\ \hline
	43 &Theorem \ref{th:kim_generalization}, Corollaries \ref{coro:r=2_v=5,13,17}, \ref{coro:union_criteria}&
	44 &Theorem \ref{th:kim_generalization}, Corollaries \ref{coro:r=2_v=5,13,17}, \ref{coro:union_criteria}\\ \hline
	45 &Corollary \ref{coro:union_criteria}&
	46 &Corollaries \ref{coro:r=2_v=5,13,17}, \ref{coro:union_criteria}\\ \hline
	47 &Theorem \ref{th:kim_generalization} &
	48 &Theorem \ref{th:kim_generalization}, Corollaries \ref{coro:r=2_v=5,13,17}, \ref{coro:union_criteria}\\ \hline
	49 &Corollary \ref{coro:r=2_v=5,13,17}&
	50 &Theorem \ref{th:kim_generalization} \\ \hline
	51 &Theorem \ref{th:kim_generalization}, Corollary \ref{coro:union_criteria}&
	52 &Theorem \ref{th:kim_generalization}, Corollary \ref{coro:union_criteria}\\ \hline
	53 &Theorem \ref{th:kim_generalization}, Corollary \ref{coro:r=2_v=5,13,17}&
	54 &Theorem \ref{th:kim_generalization}, Corollary \ref{coro:r=2_v=5,13,17}\\ \hline
	55 &?&
	56 &Theorem \ref{th:kim_generalization}, Corollaries \ref{coro:r=2_v=5,13,17}, \ref{coro:union_criteria}\\ \hline
	57 &Theorem \ref{th:kim_generalization}, Corollaries \ref{coro:r=2_v=5,13,17}, \ref{coro:union_criteria}&
	58 &Corollary \ref{coro:r=2_v=5,13,17}\\ \hline
	59 &Corollary \ref{coro:union_criteria}&
	60 &Theorem \ref{th:kim_generalization} \\ \hline
	61 &Corollaries \ref{coro:r=2_v=5,13,17}, \ref{coro:union_criteria}&
	62 &Theorem \ref{th:kim_generalization}, Corollaries \ref{coro:r=2_v=5,13,17}, \ref{coro:union_criteria}\\ \hline
	63 &Theorem \ref{th:kim_generalization}, Corollaries \ref{coro:r=2_v=5,13,17}, \ref{coro:union_criteria}&
	64 &?\\ \hline
	65 &Theorem \ref{th:kim_generalization}, Corollary \ref{coro:union_criteria}&
	66 &?\\ \hline
	67 &Theorem \ref{th:kim_generalization}, Corollaries \ref{coro:r=2_v=5,13,17}, \ref{coro:union_criteria}&
	68 &Theorem \ref{th:kim_generalization}, Corollary \ref{coro:r=2_v=5,13,17}\\ \hline
	69 &Theorem \ref{th:kim_generalization} &
	70 &Theorem \ref{th:kim_generalization} \\ \hline
	71 &Theorem \ref{th:kim_generalization}, Corollary \ref{coro:r=2_v=5,13,17}&
	72 &Theorem \ref{th:kim_generalization} \\ \hline
	73 &Theorem \ref{th:kim_generalization}, Corollaries \ref{coro:r=2_v=5,13,17}, \ref{coro:union_criteria}&
	74 &Theorem \ref{th:kim_generalization}, Corollaries \ref{coro:r=2_v=5,13,17}, \ref{coro:union_criteria}\\ \hline
	75 &Theorem \ref{th:kim_generalization}, Corollaries \ref{coro:r=2_v=5,13,17}, \ref{coro:union_criteria}&
	76 &Theorem \ref{th:kim_generalization} \\ \hline
	77 &Corollary \ref{coro:union_criteria}&
	78 &?\\ \hline
	79 &Theorem \ref{th:kim_generalization} &
	80 &Theorem \ref{th:kim_generalization} \\ \hline
	81 &Theorem \ref{th:kim_generalization}, Corollaries \ref{coro:r=2_v=5,13,17}, \ref{coro:union_criteria}&
	82 &Theorem \ref{th:kim_generalization} \\ \hline
	83 &Theorem \ref{th:kim_generalization}, Corollary \ref{coro:r=2_v=5,13,17}&
	84 &Theorem \ref{th:kim_generalization} \\ \hline
	85 &Theorem \ref{th:kim_generalization}, Corollary \ref{coro:union_criteria}&
	86 &Corollary \ref{coro:r=2_v=5,13,17}\\ \hline
	87 &Theorem \ref{th:kim_generalization} &
	88 &Theorem \ref{th:kim_generalization}, Corollary \ref{coro:r=2_v=5,13,17}\\ \hline
	89 &Theorem \ref{th:kim_generalization} &
	90 &Theorem \ref{th:kim_generalization}, Corollary \ref{coro:union_criteria}\\ \hline
	91 &Corollary \ref{coro:union_criteria}&
	92 &?\\ \hline
	93 &Theorem \ref{th:kim_generalization}, Corollaries \ref{coro:r=2_v=5,13,17}, \ref{coro:union_criteria}&
	94 &Theorem \ref{th:kim_generalization}, Corollary \ref{coro:union_criteria}\\ \hline
	95 &Corollaries \ref{coro:r=2_v=5,13,17}, \ref{coro:union_criteria}&
	96 &Corollary \ref{coro:r=2_v=5,13,17}\\ \hline
	97 &Theorem \ref{th:kim_generalization} &
	98 &Corollary \ref{coro:r=2_v=5,13,17}\\ \hline
	99 &Theorem \ref{th:kim_generalization} &
	100 &Theorem \ref{th:kim_generalization} \\ \hline
\end{longtable}

\end{document}